\documentclass[11pt,a4paper]{article}

\usepackage[english]{babel}
\usepackage{   amssymb,amsfonts,amsthm,enumitem,stmaryrd,
		      tikz-cd,mathrsfs,etoolbox,bbm,
		      extarrows
			}
        \usepackage[letterspace = 125]{microtype}
\usepackage[left = 2.7cm,right=2.7cm,top = 3.5cm,bottom=3.5cm]{geometry}

\usepackage[colorlinks=false]{hyperref}
\usetikzlibrary{quotes,babel,angles,backgrounds,fit,calc}
\usepackage{multicol}

\theoremstyle{plain}
\newtheorem{thm}{Theorem}[section]
\newtheorem{prop}[thm]{Proposition}
\newtheorem{cor}[thm]{Corollary}
\newtheorem{lem}[thm]{Lemma}

\theoremstyle{definition}
\newtheorem{defi}[thm]{Definition}
\newtheorem{exas}[thm]{Examples}

\theoremstyle{remark}
\newtheorem{rem}[thm]{Remark}

\DeclareMathOperator{\CM}{\mathbf{CM}}

\DeclareMathOperator{\dHom}{dHom}

\DeclareMathOperator{\dom}{dom}
\DeclareMathOperator{\Eq}{Eq}

\DeclareMathOperator{\Equi}{Equi}
\DeclareMathOperator{\eval}{\mathtt{eval}}
\DeclareMathOperator{\fHom}{fHom}

\DeclareMathOperator{\Fun}{Fun}
\newcommand{\Grothendieck}[2]{\sum_{#1}#2}
\DeclareMathOperator{\Hom}{Hom}

\DeclareMathOperator{\id}{id}

\DeclareMathOperator{\Mon}{Mon}
\DeclareMathOperator{\none}{\phantom{\cdot}}
\DeclareMathOperator{\opp}{\!^{\mathrm{op}}}
\DeclareMathOperator{\pr}{\mathsf{pr}}

\DeclareMathOperator{\pull}{\mathtt{pull}}
\DeclareMathOperator{\pullback}{\ar[dr,"\lrcorner", description, phantom, at start]}

\DeclareMathOperator{\simto}{\rightarrowtriangle}

\DeclareMathOperator{\To}{\Rightarrow}

\renewcommand{\textminus}{\,\text{-}\,}

\newcommand{\bbone}{\text{\usefont{U}{bbold}{m}{n}1}}
\MakeRobust{\bbone}
\newcommand{\bbtwo}{\text{\usefont{U}{bbold}{m}{n}2}}
\MakeRobust{\bbtwo}
\newcommand{\bbthree}{\text{\usefont{U}{bbold}{m}{n}3}}
\MakeRobust{\bbthree}
\newcommand{\bbfour}{\text{\usefont{U}{bbold}{m}{n}4}}
\MakeRobust{\bbfour}
\newcommand{\bbfive}{\text{\usefont{U}{bbold}{m}{n}5}}
\MakeRobust{\bbfive}
\newcommand{\bbsix}{\text{\usefont{U}{bbold}{m}{n}6}}
\MakeRobust{\bbsix}
\newcommand{\bbseven}{\text{\usefont{U}{bbold}{m}{n}7}}
\MakeRobust{\bbseven}
\newcommand{\bbeight}{\text{\usefont{U}{bbold}{m}{n}8}}
\MakeRobust{\bbeight}
\newcommand{\bbnine}{\text{\usefont{U}{bbold}{m}{n}9}}
\MakeRobust{\bbnine}
\newcommand{\bbnull}{\text{\usefont{U}{bbold}{m}{n}0}}
\MakeRobust{\bbnull}

\DeclareMathOperator{\Cat}{\mathbf{Cat}}
\DeclareMathOperator{\CatBaseComp}{\mathsf{CatBaseComp}}

\DeclareMathOperator{\CompMod}{\mathsf{CompMod}}

\DeclareMathOperator{\Dialg}{\mathsf{DiAlg}}

\DeclareMathOperator{\Par}{\mathsf{Par}}
\DeclareMathOperator{\rCat}{\mathsf{rCat}}

\DeclareMathOperator{\Sets}{\mathsf{Sets}}
\DeclareMathOperator{\SetsB}{\mathbf{Sets}}

\newcommand{\C}[1]{\mathscr{#1}}
\newcommand{\B}[1]{\mathbf{#1}}

\newcommand{\CB}[1]{(\C{#1},#1)}

\tikzcdset{simto/.tip={Glyph[glyph math command=rightarrowtriangle]}}
\usepackage[backend = bibtex]{biblatex}
\newlength{\boxi}
\newlength{\boxilaenge}
\newlength{\boxii}
\newlength{\abstractwdth}
\newlength{\keywordwdth}
\newlength{\authorwdth}
\newlength{\Infowdth}
\newlength{\hanginglength}

\newcommand{\titlesetting}[5]{%
	\hrule%
	\kern25pt%
	\hbox to \textwidth{\vbox{#1}}%
	\kern25pt%
	\hrule%
	\hbox to \textwidth{%
		\vtop{%
			\kern 10pt%
            \hbox to .65\textwidth{%
				\kern10pt\vtop{\hsize\boxii \noindent\textls{Abstract}%
                \par\vskip-.5em\noindent                \rule{1.3\abstractwdth}{.4pt}%
                \par\noindent%
\small#2\normalsize\par\medskip\noindent\textls{Keywords}%
\par\vskip-.5em\noindent
                \rule{1.3\keywordwdth}{.4pt}%
                \par\noindent
\small#5}} \kern 10pt
		}%
        \vrule%
		\vtop{%
			\kern10pt%
            \hbox to \boxi{%
            \kern 10pt \vtop{\hsize\boxilaenge
            \noindent\textls{Authors}%
            \par\vskip-.5em\noindent\rule{1.3\authorwdth}{.4pt}%
            \par\smallskip\noindent%
            #3\normalsize}
            }
			\kern 10pt
			}
}
	\hrule%
    \hbox to \textwidth{%
		\vtop{%
			\kern 10pt\hbox to .96\textwidth{%
				\kern10pt\vtop{\hsize.96\textwidth \noindent\textls{Article Info}\par\vskip-.5em\noindent
                \rule{1.3\Infowdth}{.4pt}\par\noindent
\small#4}}\kern 10pt}}
    \hrule
}

\newcommand{\titlesetter}[6]{%
	\vbox{%
		\hrule
	\kern25pt%
	\hbox to \textwidth{\vbox{#1}}%
	\kern10pt
	\hbox to \textwidth{%
    \hfill
		\vtop{%
			\kern10pt%
			\hbox{%
                \vtop{%
            \hbox to \boxi{%
            \kern 10pt \vtop{\hsize\boxilaenge
            #3\normalsize}
}}
\vtop{%
            \hbox to \boxi{%
            \kern 10pt \vtop{\hsize\boxilaenge
            #4\normalsize}
            }
			}\hfil\phantom{.}
		}
        \kern 15pt
        }
        \hfill
}
 \hbox to \textwidth{%
        \kern10pt \vtop{%
        \hbox{#5}\normalsize
        \kern 15pt}
	       }
}
	\hrule%
\hbox to \textwidth{%
		\vtop{%
			\kern 10pt\hbox to .96\textwidth{%
				\kern10pt\vtop{\hsize.96\textwidth \noindent\textls{Abstract}\par\vskip-.5em\noindent
                \rule{1.3\abstractwdth}{.4pt}\par\noindent
\small#2}}\kern 10pt
\hbox to .96\textwidth{%
	\kern10pt\vtop{\hsize.96\textwidth %
\noindent\textls{Keywords}\par\vskip-.5em\noindent
                \rule{1.3\keywordwdth}{.4pt}\par\noindent
\small#6}}\kern 10pt
}}
\hrule
}

\begin{document}

\titlesetter{%
\centering\sc{\huge Categories with a base of computability}
}
{%
The notion of a base of computability $C$ in a category $\mathscr{C}$ was introduced in~\cite{petrakisStrictComputabilityModels2022} as a tool to generate computability models, in the sense of Longley and Normann, from categories. In this paper we introduce the category $\CatBaseComp$ of categories with a base of computability, and we show that \( \CatBaseComp \) has all pie limits. We prove that a Grothendieck fibration lifts a base of computability in the base category to a base of computability in the total category of the fibration, and conversely, a pullback-preserving Grothendieck fibration maps a base of computability in the total category to a base of computability in the base category of the fibration. Connecting $\CatBaseComp$ with the semantics of dependent type theory, we show that $\CatBaseComp$ is a type-category, or a (fam, $\Sigma)$-category with a terminal object. Moreover, we prove that $\CatBaseComp$ is a 
(2-fam, \( \Sigma \))-category, a 2-categorical generalisation of a (fam, $\Sigma)$-category. Finally, we describe the canonical 
(2-dep, $\Sigma$)-structure of $\CatBaseComp$, i.e., the canonical dependent arrows of $\CatBaseComp$ that are compatible with its (2-fam, \( \Sigma \))-structure.
}
{%
\noindent Luis \sc{Gambarte}\textsuperscript{1}\normalfont\normalsize
    \\
    \footnotesize\href{mailto:gambarte@math.lmu.de}{\ttfamily gambarte@math.lmu.de}
}{%
	\noindent\normalsize Iosif \sc{Petrakis}\textsuperscript{1}\normalfont
    \\
   \footnotesize\href{mailto:petrakis@math.lmu.de}{\ttfamily petrakis@math.lmu.de}
}{%
  \noindent\small\textsuperscript{1}\!Ludwig-Maximilians-Universität M\"unchen, Theresienstraße 39, D-80333 München}
{%
Computability models, category theory, dominions, type-category, dependent type theory
}
\normalsize

%\renewcommand{\contentsname}{\relax\vskip-1.3\baselineskip}
%\begin{multicols}{2}
%{\tableofcontents}
%\end{multicols}
%\vskip 1em
%\hrule

\section{Introduction}

In~\cite{Pe22}, the notion of a category with a \emph{base of computability} was introduced, in order to generate computability models, in the sense of Longley and Normann~\cite{LN15}, from categories (see Remark~\ref{ex: base}). 
If $T$ is a class, the objects of which are called \emph{type names}, a \emph{computability model} $\mathbf{C}$ over $T$ is a pair $(\mathbf{C}(t)_{t \in T}, (\mathbf{C}[s, t])_{s, t \in T})$, where $\mathbf{C}(t)$ is a set of \emph{data types}, for every $t \in T$, and $\mathbf{C}[s, t]$ is a class of partial functions from $\mathbf{C}(s)$ to $\mathbf{C}(t)$, such that the identity function $1_{\mathbf{C}(t)}$ is in $\mathbf{C}[t,t]$, and for every $f \in \mathbf{C}[r,s]$ and $g \in \mathbf{C}[s,t]$ we have $g \circ f \in \mathbf{C}[r,t]$.
Many notions of category theory have their counterparts in the theory of computability models. For example, to the category $\Sets$ of sets and functions corresponds the computability model $\SetsB$ over the class of sets, where $\mathbf{C}(X) := X$, for every set $X$, and $\mathbf{C}[X,Y]$ is the set of partial functions from $X$ to $Y$. In~\cite{GP24, GP25} the categorical Grothendieck construction motivated the definition and study of the so-called Grothendieck computability model.
\emph{Simulations} are the arrows between computability models.
In~\cite{GP25} it is shown that $\CompMod$, the category of computability models and simulations, is a type-category, in the sense of Pitts~\cite{Pi01}, a result that bridges the categorical interpretation of dependent type theory with the theory of computability models. In the proof of this result the Grothendieck computability model is the $Sigma$-object over a computability model $\mathbf{C}$ and a copresheaf-simulation $\B {\gamma} \colon \B C \simto \SetsB$ (see also section~\ref{sec: typecat}).

Here we extend~\cite{Pe22} introducing the category $\CatBaseComp$ with objects the categories with a base of computability, and with arrows the so-called \emph{computability transfers}. Despite the similarity of the notion of a base of computability with Rosolini's notion of dominion, introduced in~\cite{rosoliniContinuityEffectivenessTopoi1986}, the study of $\CatBaseComp$ is, to our knowledge, new. Actually, we do not know of a systematic study of the category of categories with a dominion, the arrows of which, though, are defined slightly differently from computability transfers (see section~\ref{sec: CatBaseComp}). 

In addition to showing that \( \CatBaseComp \) has all pie limits, we
describe how a Grothendieck fibration acts on a base of computability in the total category, or in the base category of the fibration.
More importantly, we show that $\CatBaseComp$ itself is directly connected to the categorical semantics of dependent type theory. For that, we first prove that $\CatBaseComp$ is a type-category, or, equivalently, a (fam, $\Sigma)$-category with a terminal object. Moreover, we show that $\CatBaseComp$ is a 
(2-fam, \( \Sigma \))-category, a 2-categorical generalisation of a (fam, $\Sigma)$-category. Finally, we describe the canonical 
(2-dep, $\Sigma$)-structure of $\CatBaseComp$, i.e., the canonical dependent arrows of $\CatBaseComp$ that are compatible with its (2-fam, \( \Sigma \))-structure.

We structure this paper as follows:

\begin{itemize}

\item In section~\ref{sec: CatBaseComp} we introduce the category $\CatBaseComp$ with objects the categories with a base of computability and arrows the computability transfers. We also briefly discuss the relation of a category with a base of computability to Rosolini's categories with a dominion, to Rosolini's $p$-categories, to dominical categories of Di Paola and Heller, and to restriction categories of Cockett and Lack.

\item In section~\ref{sec: limits} we show that the category \( \CatBaseComp \) has all set-indexed products, as well as all equifiers and inserters, and hence, it has all pie limits, as these are described in~\cite{powerCharacterizationPieLimits1991}.

\item In section~\ref{sec: fibr} we show that a Grothendieck fibration 
lifts a base of computability in the base category to a base of computability in the total category of the fibration (Proposition~\ref{prp: FibLiftBase}). Furthermore, using Grothendieck fibrations we regain some pullbacks in $\CatBaseComp$ (Theorem~\ref{thm: PullbackFibLift}), which do not exist, in general, as it is explained in section~\ref{sec: limits}.

\item In section~\ref{sec: transport} we define various new bases of computability from given ones. Namely, we show that a pullback-preserving fibration $F \colon \mathscr{E \to B}$ maps a base of computability $E$ in $\mathscr{E}$ to a base of computability $F(E)$ in $\mathscr{B}$ (Proposition~\ref{prp: basetobase}). Moreover, we describe the base of computability in the comma-category $F / G$ (Proposition~\ref{prp: commabase}), and the exponential in a natural subcategory of $\CatBaseComp$ (Proposition~\ref{prp: exp}).

\item In section~\ref{sec: typecat}, we present the notion of a \emph{category with a family-arrow structure} and \emph{Sigma-objects}, introduced in~\cite{petrakisCategoriesDependentArrows2023} and directly linked to the notion of \emph{type-category}, introduced by Pitts in~\cite{Pi01}. Sigma-objects generalise the Grothendieck construction in the abstract framework of categories with a family-arrow structure. The main result of this section is that the category $\CatBaseComp$ is a type-category, i.e., a (fam, $\Sigma$)-category with a terminal object (Theorem~\ref{thm: typecat}). In this way, the category $\CatBaseComp$ can be seen as a model of dependent type theory. 

\item In section~\ref{sec: 2famS} we show that \( \CatBaseComp \) is a (2-fam, \( \Sigma \))-category (Proposition~\ref{prp: 2famSigma}), a 2-categorical generalisation of a (fam, $\Sigma$)-category, studied by Ehrhardt in~\cite{ehrhardt2depCategories2024}.

\item A dependent arrow is an abstract categorical formulation of the type-theoretic notion of a dependent function. The axioms of a category endowed with a dependent-arrow structure capture the properties of composition of a dependent function with a function, exactly as the axioms of a category capture the properties of compositions of functions.
In \cite{petrakisCategoriesDependentArrows2023} it is shown that a (fam, $\Sigma$)-category can be equipped with a canonical dependent arrow-structure, turning it into a (dep, $\Sigma$)-category.
In section~\ref{sec: 2depS} we present the notion of a (2-dep, $\Sigma$)-category, and describe the dependent arrows obtained by the canonical construction, applied to the (2-fam, $\Sigma$)-structure on \( \CatBaseComp \), presented in section~\ref{sec: 2famS}.

%\item In section~\ref{sec: canonical} we 

\end{itemize}

For all notions and results from category theory that we use here without explanation or proof, we refer to~\cite{Ri16}. For all notions and results from the theory of computability models that we use here without explanation or proof, we refer to~\cite{LN15, Pe22, GP24, GP25}.
We use bold letters to denote a computability model.
This paper is a continuation of~\cite{Pe22}, published in this journal.

\section{The category \texorpdfstring{$\CatBaseComp$}{CatBaseComp}}\label{sec: CatBaseComp}

A base of computability in a category $\C{C}$
%, introduced in~\cite{Pe22}, 
is a family of monos, indexed by the objects of $\C{C}$, which induces in a simple and natural way canonical computability models. It is motivated exactly by this direct generation of computability models from the category $\C{C}$ equipped with a base. \textit{Throughout this paper $\C{C, D}$ are categories}.

\begin{defi}[Base of computability]\label{def: base}
%Let $\mathscr{C}$ be a category. 
A \emph{base of computability} $B$ in $\mathscr{C}$ is a family
$(B(a))_{a \in \mathscr{C}_0} $ of subclasses $B(a) \subseteq \Mon(\textminus, a)$, such that the following conditions hold:\\[1mm]
(\texttt{Base}\textsubscript{1}) $1_a \in B(a)$, for every $a \in \mathscr{C}$.\\[1mm]
(\texttt{Base}\textsubscript{2}) For every $a,b \in \mathscr{C}$ and $i \colon s \to a, j \colon t \to b$, such that
	$i \in B(a), j \in B(b)$, and $f \colon s \to b$,
	a pullback $j^*(s)$ exists, and the composite arrow $i \circ f^\ast j $ is in $B(a)$
	\[ \begin{tikzcd}
	& j^*(s)\ar[r,"j^\ast f"] \ar[d,"f^\ast j",hook] \ar[dl,"i \circ  f^\ast j"',hook] \pullback & t \ar[d,"j",hook]
	\\
	a & s \ar[l,"i",hook] \ar[r,"f"'] & b.
	\end{tikzcd} \]
If the base of computability fulfills the following stronger condition, we call it \emph{replete}.\\[1mm]
(\texttt{Replete}) For every $a,b \in \mathscr{C}$ and $i \colon s \to a, j \colon t \to b$, such that
	$i \in B(a), j \in B(b)$, and $f \colon s \to b$,
	a pullback $j^*(s)$ exists, and for every such pullback the composite arrow $i \circ f^\ast j$ is in $B(a)$
	\[ \begin{tikzcd}
	& j^*(s)\ar[r,"j^\ast f"] \ar[d,"f^\ast j",hook] \ar[dl,"i \circ  f^\ast j"',hook] \pullback & t \ar[d,"j",hook]
	\\
	a & s \ar[l,"i",hook] \ar[r,"f"'] & b.
	\end{tikzcd} \]
%	commutes and $i \circ f^\ast j \in B(a)$.
\end{defi}

\begin{exas}
    \label{ex:bases}
    % Identity bases, isomorphisms, all monos Paper (22) 
    %% Done
    If $\C C$ is a category, one can form the \emph{identity base} $I$ which contains only the identities, that is $I(c) = \{1_c\}$ for every $c \in \C C$.
    The pullbacks necessary for condition (\texttt{Base}\textsubscript{2}) are of the following form:
    \[ \begin{tikzcd}
        d \ar[d,"1_d"'] \ar[r,"f"] &c \ar[d,"1_c"] 
        \\
        d \ar[r,"f"'] & c.
    \end{tikzcd}\]
    The \emph{isomorphism base} $\mathrm{Iso}$ contains all isomorphisms, that is $\mathrm{Iso}(c) = \{i \colon d \xlongrightarrow{\cong} c ~\vert~ i \text{ iso}\}$. 
    The pullbacks necessary for condition (\texttt{Base}\textsubscript{2}) are of the following form:
    \[ \begin{tikzcd}
        d \ar[d,"1_d"'] \ar[r," i^{-1} \circ f"] &c' \ar[d,"i"] 
        \\
        d \ar[r,"f"'] & c.
    \end{tikzcd}\]
    This base is replete, as if another pullbacks exists, it is connected via a unique mediating iso to a pullback of the above form, but all such isomorphisms are in the isomorphism base.

    If $\C C$ has pullbacks, the \emph{partial base} $\mathrm{Mon}$ contains all monos, that is $\mathrm{Mon}(c) = \{m \colon d \hookrightarrow c~\vert~ m \text{ mono}\}$. 
    This base is also replete.
\end{exas}

\begin{rem}
    One can show that a replete base of computability is closed under isos---in the sense that $B(a)$ contains all isomorphisms with codomain $a$---and compositions, and conversely any base of computability that is closed under composition and isos is replete.% Furhter elaboration
\end{rem}

\begin{rem}\label{ex: base}
Next, we explain why we use the term \emph{base of computability} in Definition~\ref{def: base}.
Following~\cite{Pe22}, if $\mathscr{C}$ has pullbacks and $S \colon \mathscr{C} \to \Sets$ is a pullback-preserving presheaf, then the family $B$ with $B(a) := \Mon(\textminus, a)$, for every $a \in \mathscr{C}$, is a base of computability in $\C{C}$. The \emph{partial computability model} $\CM^{B}(\mathscr{C};S)$ induced by $B$ is the pair $\big(S(a)_{a \in \C{C}}, (S_B[a,b])_{a, b \in \C{C}}\big)$, where 
$$S_B[a,b] := \{(S(i), S(f)) \mid i \in B(a) \ \& \ f \in \Hom(\dom(i), b)\}.$$
Recall that if $(i, f)$ is a partial arrow $a \rightharpoonup b$ in $\C{C}$, then $(S(i), S(f))$ is a partial function $S(a) \rightharpoonup S(b)$, where 
$$\dom(S(i), S(f)) := S(\dom(i)),$$
$$[(S(i), S(f)](S(i)(x)) := S(f)(x).$$
If $B(a) := \{1_a\}$, for every $a \in \C{C}$, then the above computability model is reduced to the \emph{total computability model},
%$\CM^{\total}(\mathscr{C};S)$, 
where the class of its type names is the class of the objects of $\C{C}$ and its data types are the sets $S(a)$, for every $a \in \C{C}$.
\end{rem}
%\begin{exas}\hfill
%\begin{enumerate}[leftmargin = 2em]
%	\item Let $\mathscr{C}$ be an arbitrary category with a copresheaf $S \colon \mathscr{C} \to
%	\Sets$. We consider the base $\total$ on $\mathscr{C}$ (also called \emph{identity base}), defined
%through $\total(a) = \{\mathbf{1}_a\}$ for all $a \in \mathscr{C}$. 
%	The computability model	 is called the \emph{total computability
%	model} associated to $\mathscr{C}$ and $S$.
%	\item Let $\mathscr{C}$ be a category with all pullbacks and $S \colon \mathscr{C} \to
%	\Sets$ be a pullback-preserving presheaf. Let $\prt$ be the computability model defined by
%	$\prt(a) = \Mon(\textminus, a)$ for all $a \in \mathscr{C}$. The model $\CM^{\prt}(\mathscr{C};S)$ is called the \emph{partial computability
%	model} associated to $\mathscr{C}$ and $S$.\end{enumerate}
%\end{exas}

%%%%%%%%%%%%%%%%%%%%

\begin{rem}
    \label{rem: dominion}
The notion of \emph{dominion}, introduced by Rosolini in \cite{rosoliniContinuityEffectivenessTopoi1986}, is
 similar to that of a base of computability. 
If $\C{C}$ has binary products and $\mathcal{M}$ is a family of monos in $\C{C}$, closed under identity and composition, such that any pullback of a mono in $\mathcal{M}$
exists in $\C{C}$ and a representative of it belongs to $\mathcal{M}$, then $\mathcal{M}$ is a dominion.
The difference between a base of a computability and a dominion is that the monos in a dominion form a wide subcategory of the category $\C C$, whereas the monos in a base of computability do not: they must not be closed under composition, as this would require that for any $m \in B(a)$ for arbitrary $a$ we have that the pullback 
\[
\begin{tikzcd}
    c \ar[d,"m"'] \ar[r,"1_c"] \pullback & c \ar[d,"m"]
    \\
    a \ar[r,"1_a"] & a
\end{tikzcd}
\]
is the one used in (\texttt{Base}\textsubscript{2}), but this is not guaranteed.

However, replete bases of computability are closed under composition, as remarked earlier, as condition (\texttt{Replete}) entails closure under all possible choices of pullbacks, in particular for the ones of the above form.
Hence, replete bases of computability are dominions, albeit not necessarily in categories with binary products.

 If $\C{D}$ has products and a dominion $\mathcal{N}$, a functor $F \colon
\C{C\to D}$ that takes monos in $\mathcal{M}$ to monos in $\mathcal{N}$ is said to \emph{preserve
dominions} (see~\cite{rosoliniContinuityEffectivenessTopoi1986}). This notion differs from our definition of a computability transfer, as it does not require to preserve pullbacks.
\end{rem}

%%%%%%%%%%%%%%%%%%%%

% Insert exas here

%%%%%%%%%%%%%%%%%%%%%%%%%%%%%%%%%%%%%%%%

% Mention cobases of computability?

%%%%%%%%%%%%%%%%%%%%%%%%%%%%%%%%%%%%%%%%%
% Comparison with the notion of dominion

%%%%%%%%%%%%%%%%%%%%%%%%%%%%%%%%%%%%%%%%%
% Comparison with the notion representation of partial morphisms in acc.pdf
 	\emph{From now on, we usually use the same letter for the base of computability in a category, but in normal serif font, and \( (\C C, C)\), \((\C D, D)\), and \((\C E, E)\) are categories with a base of computability.}

\begin{defi}[Computability transfers]
A \emph{computability
transfer} $F \colon (\mathscr{C},C) \to (\mathscr{D},D)$ is a  functor $F \colon \mathscr{C \to D}$, such that
for every $a \in \mathscr{C}$ we have $F\big(C(a)\big) = \big\{ F(i) ~\vert~ i \in C(a)\big\} \subseteq D\big(F(a)\big)$, and $F$ preserves all pullbacks mentioned in (\texttt{Base}\textsubscript{2}).
\end{defi}

% Add examples 
%% Done
\begin{exas}
    \label{exas: CompTrafos}
    For the identity bases, every functor $F \colon \C C \to \C D$ is a computability transfer $F \colon (\C C,I) \to (\C D,I)$.
    Similarly, every such functor $F \colon (\C C,\mathrm{Iso}) \to (\C D, \mathrm{Iso})$ is a computability transfer,
    and if $\C C, \C D$ have pullbacks and $F$ preserves them, then $F$ is a computability transfer $F \colon (\C C, \mathrm{Mon}) \to (\C D, \mathrm{Mon})$.
\end{exas}

\begin{defi}[The 2-category $\CatBaseComp$]
Let $\CatBaseComp$ be the 2-category with
\emph{objects} the pairs $(\mathscr{C},C)$, where
	$\mathscr{C}$ is a category and $C$ is a base of computability in $\mathscr{C}$, \emph{1-cells} $F \colon (\mathscr{C},C) \to (\mathscr{D},D)$ computability transfers, and
\emph{2-cells} natural transformations \( F \To G \) between the underlying functors.
\end{defi}

\begin{rem}\label{rem: comparison}
\normalfont 
As explained in Remark~\ref{ex: base}, a base of computability in $\C{C}$ allows the definition of a computability model and its partial computable functions, in a canonical way. It is not an accident that the notion of a base of computability is close to Rosolini's dominions and $p$-categories, introduced in~\cite{rosoliniContinuityEffectivenessTopoi1986}, to dominical categories of Di Paola and Heller~\cite{dipaolaDominicalCategoriesRecursion1987}, and to restriction categories, introduced by Cockett and Lack in  \cite{cockettRestrictionCategoriesCategories2002}. Clearly, a category \( \C C \) with base of computability \( C \) induces a category \( \Par(\C C,C) \) of partial arrows.
It is straightforward to show that if \( \C C \) has products and an initial object, then \( \Par(\C C,C) \) is a dominical category.
In~\cite[Proposition 2.2.1]{rosoliniContinuityEffectivenessTopoi1986} it is shown that every \( p \)-category is a dominical category, while the inverse does not hold, in general. Moreover, one can turn any \( p \)-category into a restriction category (for all necessary details related to these notions, we also refer to~\cite{Ga26}).
%y defining 
%\[
	% \overline{f} := p_{c,d} \circ (1_c \times f) \circ \Delta_c 
%\]
%for \( f \colon c \to d \).
Hence, studying categories with a base of computability gives the impression that we are working at the lowest possible level of abstraction, and it would be wiser to consider restriction categories instead.
%of categories with a base of computability.
However, within our language, Cockett and Lack showed in~\cite{cockettRestrictionCategoriesCategories2002} that there is an adjunction 
\[ \begin{tikzcd}
	\CatBaseComp \ar[r,shift right = 1.5ex,""{name = U}] & \rCat \ar[l,shift right = 1.5ex,""{name = V}]
	\ar[from = U, to = V,"\vdash"{sloped, marking}, phantom]
\end{tikzcd} \]
where \( \rCat \) is the category of restriction categories and restriction functors, i.e., functors preserving the restriction structure.
Furthermore, the above functor  \( \CatBaseComp \to \rCat \) is fully faithful.
%Additionally, Longley provedI
Moreover, using our language, Cockett and Lack proved in~\cite{cockettRestrictionCategoriesCategories2002}, that if one works with restriction categories where all restriction idempotents
%\( \overline{f} \) 
split, then one has a 2-equivalence between \( \CatBaseComp \) and the category \( \rCat_s \) of these split restriction categories.
\end{rem}

\section{\texorpdfstring{$\CatBaseComp$}{CatBaseComp} has all pie limits}\label{sec: limits}
In this section we show that the category \( \CatBaseComp \) has all set-indexed products, as well as all equifiers and inserters, and hence, it has all pie limits, as these are described in~\cite{powerCharacterizationPieLimits1991}. 
Pie limits are defined in~\cite{powerCharacterizationPieLimits1991} as a special kind of weighted limit for 2-categories.

\begin{defi}\label{def: wlimits}
%[Weighted limits]
	Let \( \C C, \C W \) be 2-categories and \( F \colon \C W \to \Cat, G \colon \C W \to \C C \) be 2-functors.
	An \emph{\( F \)-weighted limit} of \( G \) is a representing object \( \lim^FG \) of 
	\[
		[\C W, \Cat]\Big(F, \C C\big(\textminus,G(\textminus)\big)\Big),
	\]
	that is, for every \( C \in \C C \) there is an isomorphism of categories, natural in $\C{C}$
	\[
		\C C[C,\lim\nolimits^F G] \cong [\C W,\Cat] \Big(F,\C C\big(C,G(\textminus)\big)\Big).
	\]
	%natural in \( C \).
\end{defi}

A textbook-account of weighted limits is found in~\cite[Chap.~6]{borceuxHandbookCategoricalAlgebra1994}, where 
%this book 
weighted limits are introduced for arbitrary categories enriched over a symmetric monoidal closed category \( \C V\). 
Here, we only require \( \C V = \Cat \).

\begin{exas}\label{ex: wlimits}
%\hfill
%	\begin{enumerate} 
	%\item
    \emph{Products} are weighted limits where \( \C W \) is a discrete 2-category on a set \(I\). 
			The weight \( F \) maps all objects \( i \) to the terminal category \( \bbone \).\\
		%\item
        \emph{Inserters} are weighted limits where \( \C W\) is the 2-category given by
		\[ \begin{tikzcd}
			0 \ar[r,"f_0"{name = U}, shift left = .5ex] \ar[r,"f_1"{name = V, swap}, shift right = .5ex] & 1.
		\end{tikzcd} \]
		The weight \( F \) takes \( 0 \) to the terminal category \( \bbone \) and \( 1 \) to the 2-category \( \Delta(1) \) with two objects and only the non-trivial arrow \( 0\to1 \), i.e., \(\Delta(1)\) is the 1-simplex. All 2-cells in \( \Delta(1) \) are trivial.
		The arrow \( f_0 \) is mapped to the functor taking \( 0  \) to \( 0 \) and the identity of \( 0 \) to itself, and the arrow \( f_1 \) is mapped to the functor taking \( 0 \) to \( 1 \) and the identity to the identity.\\
%		\item
        \emph{Equifiers} are weighted limits where \( \C W \) is the 2-category given by
			\[ \begin{tikzcd}
				0 \ar[r,"f_0"{name = U}, shift left = 1.5ex] \ar[r,"f_1"{name = V, swap}, shift right = 1.5ex] &[1em] 1.
				\ar[from = U, to = V, "\gamma_1", shift left = .75ex, Rightarrow,shorten = 1pt] \ar[from = U, to = V, "\gamma_0"', shift right = .75ex, Rightarrow, shorten = 1pt]
			\end{tikzcd} \]
			The weight \( F \) takes \( 0 \) to the terminal category \( \bbone\) and \( 1 \) to the 2-category \( \Delta(1) \), defined above.
	%\end{enumerate}
\end{exas}

\begin{rem}\label{rem: Rob}
In~\cite{powerCharacterizationPieLimits1991} Power and Robinson give an explicit description of these weighted limits in elementary terms.
%in \cite{powerCharacterizationPieLimits1991}.
%\begin{enumerate}
%\renewcommand{\labelenumi}{\theenumi.}
%\item 
Products are products in the standard categorical sense
%with their universal property 
together with the following 2-dimensional universal property: given two 1-cells (morphisms) \( f,g \colon A \to \prod_{i \in I} A_i \), the 2-cells \( \eta \colon f \To g \) are in one-to-one-correspondence with families of 2-cells \( \eta_i \colon \pr_i \circ f \To \pr_i \circ g \).
%\item 
An inserter of a diagram
	\[ \begin{tikzcd}
		A \ar[r,"f", shift left = .5ex] \ar[r,"g"', shift right = .5ex] & B
	\end{tikzcd} \]
	in a 2-category \( \C C \) consists of an object \( I \) of \( \C C \) together with a morphism \( i \colon I \to A \) and a 2-cell
 \( \alpha \colon f \circ i \To g \circ i \), which is universal with this property. Namely, \\[1mm]
 %the following properties are satisfied:
 %\begin{itemize} 
 %\item
 \emph{1-dimensional part:} For every \( J \) with \( j \colon J \to A, \beta \colon f \circ j \To g \circ j \) there is a unique \( h \colon J \to I \) making the obvious triangle commutative.\\
 \emph{2-dimensional part:} For every \( J,J' \) with \( j,\beta,j',\beta' \) as above and a 2-cell \( \nu \colon j \to j' \), such that \( \beta' \circ (f \ast \nu) = (g \ast\nu) \circ \beta  \), that is the following diagram commutes
\[ \begin{tikzcd}
	f \circ j \ar[r,"\beta",Rightarrow] \ar[d,"f \ast \nu"',Rightarrow] & g \circ j \ar[d,"g \ast \nu",Rightarrow]
	\\
	f \circ j' \ar[r,"\beta'"',Rightarrow] & g \circ j',
\end{tikzcd} \]
%commutes,
there is a 2-cell \( \mu \colon h \to h' \), such that \( i \ast \mu = \nu \).\\[1mm]
 %\end{itemize}
%\item 
An equifier of a diagram
	\[ \begin{tikzcd}
		A \ar[r,"f"{name = U}, shift left = 1.7ex] \ar[r,"g"{swap, name = V}, shift right = 1.7ex] &[1em] B
		\ar[from = U, to = V, "\gamma_0"', shift right = .75ex,Rightarrow,shorten=1pt] \ar[from = U, to = V, "\gamma_1", shift left = .75ex,Rightarrow,shorten=1pt]
	\end{tikzcd} \]
	in a 2-category \( \C C \) consists of an object \( I \) of \( \C C\)  together with an arrow \( i \colon I \to A \), such that \( \gamma_1 \ast i = \gamma_0 \ast i \), which is universal with this property. Namely,\\[1mm]
	%\begin{itemize}
		\emph{1-dimensional part:} For every \( J \) with \( j \colon J \to A \), such that \( \gamma_0 \ast j = \gamma_1 \ast j \), there is \( h \colon J \to I  \) with \( j = i \circ h \).\\
		\emph{2-dimensional part:} For every \( J,J' \) with \( j,j' \) as above and \( \nu \colon j \To j' \), there is \( \mu \colon h \to h' \), such that \( i \ast \mu = \nu \).
%	\end{itemize}
\end{rem}

\begin{exas} 
The category of categories has inserters and equifiers. 
%\begin{enumerate}
%\renewcommand{\labelenumi}{\theenumi.}
%\item 
Given functors \( F,G \colon \C C \to \C D \), the inserter of \( F \) and \( G \) is given by its \emph{category of dialgebras}---introduced by Lambek in \cite{lambekSubequalizers1970} as the \emph{subequalizing category} of \( F \) and \( G \)---\( \Dialg(F,G) \).
%	\begin{itemize}
		%\item 
        The objects of this category are pairs \( (x,\gamma) \) of arrows \( \gamma \colon F(x) \to G(x) \).
		%\item
        The arrows \( (x,\gamma) \to (y,\delta) \) are 
      %  given by maps
        arrows \( f \colon x \to y \), such that the following rectangle commutes
			\[ \begin{tikzcd}
				F(x) \ar[r,"\gamma"] \ar[d,"F(f)"'] & G(x) \ar[d,"G(f)"] 
				\\
				F(y) \ar[r,"\delta"'] & G(y).
			\end{tikzcd} \]
	%		commutes.
	%\end{itemize}
	The functor \( i \colon \Dialg(F,G) \to \C C \) maps \( (x,\gamma) \) to \( x \) and \( f \) to \( f \).
%\item 
Given functors \( F,G \colon \C C \to \C D \) and natural transformations \( \alpha,\beta \colon F \To G \), the equifier \( \Eq(\alpha,\beta) \) of \( \alpha \) and \( \beta \) is 
%given as
the category 
	%\begin{itemize}
		%\item 
       with objects the objects \( c  \) of \( \C C \), such that \( \alpha_c = \beta_c \), and 
	 arrows the arrows \( f \colon c \to c' \), such that \( c,c' \) are as above.
%	\end{itemize}
	The functor \( i \colon \Eq(\alpha,\beta) \) maps \( c \) to \( c \) and \( f  \) to \( f \).
%\end{enumerate}
\end{exas}

\begin{lem}\label{lem: prod}
	If \( (\C C_i,C_i) \) is a family of categories with bases of computability indexed by some set \( I \),
	then the product \( \prod_{i \in I} (\C C_i,C_i) \) is given by 
	\[
    \prod_{i \in I} (\C C_i,C_i) = \Big(\prod_{i \in I} \C C_i, \prod_{i \in I} C_i\Big), \text{ where }
		\Big(\prod_{i \in I} C_i\Big)(c_i)_{i \in I} = \prod_{i \in I} C_i(c_i). 
	\]
\end{lem}

\begin{proof}
	Clearly, the product-base contains all identities, as \( 1_{c_i} \in C_i(c_i) \) for all \(  c_i \). 
	For the second condition observe that pullbacks in the product category are simply the product of all the pullbacks, which exists in every category $\C C_i$ by assumption.
	That this is a product in the 2-limit sense, is immediate to show.
\end{proof}

\begin{lem}\label{lem: inserter}
	If 
    %\( \CB C, \CB D \) are categories with bases of computability and
    \( F,G \colon \CB C \to \CB D \) are computability transfers, 
    then the inserter of \( F,G \) is 
    %given by
    \( \big( \Dialg(F,G),\Dialg(F,G;C)\big)\), where \( \Dialg(F,G)\) is the category of dialgebras on \( F,G \) and 
    %\( \Dialg(F,G;C) \) is defined via
	\[
		\Dialg(F,G;C)\big(c,\alpha) \big) := \bigcup_{\beta \in \Dialg(F,G)}\left\{ 
f \in \Dialg(F,G)(\beta,\alpha) \middle\vert f \in C(c)
			\right\}.
	\]
	The computability transfer \( X \colon \big( \Dialg(F,G), \Dialg(F,G;C)\big) \to \CB C \) takes \((c,\alpha)\) to \( c \) and \( f \) to \( f \). 
	The natural transformation \( F \To G \) is the same as for the inserter in the standard categorical setting.
\end{lem}

\begin{proof} 
	Clearly, \( \Dialg(F,G;C) \) is closed under identities.
	It remains to show the second condition.
	For this, let \(f \colon (c,\beta) \to (c',\alpha) \), \( g \colon (c'',\gamma) \to (c',\alpha) \) and \( h \colon (c'',\gamma) \to (c''',\delta)\),  such that $f \in \Dialg(F,G;C)(c',\alpha)$ and $h \in \Dialg(F,G;C)(c''',\delta)$ be given. 
	To show condition (\texttt{Base}\textsubscript{2}), we first note that we obtain a pullback square in \( \C C \)
	\[ \begin{tikzcd}
		&f^*(c'') \ar[dl,"h \circ f^*g"'] \ar[r,"f^*g"] \ar[d,"g^*f"'] \pullback & c \ar[d,"f"]
		\\
		c''' &c'' \ar[l,"h"]\ar[r,"g"'] & c'.
	\end{tikzcd} \]
	%in \( \C C \). 
	As \( F,G \) are computability transfers, they preserve pullbacks of arrows in \( C \), and thus, we get the following cube
	\[ \begin{tikzcd}
F\big(f^*(c'')\big) \ar[dr,"\xi",dotted] \ar[rr,"F(f^*g)"] \ar[dd,"F(g^*f)"] && F(c) \ar[dr,"\beta"] \ar[dd,"F(f)"{near end}] 
	\\
											  &G\big((f^*(c'')\big) \pullback \ar[rr,"G(f^*g)"{near start}, crossing over] && G(c) \ar[dd,"G(f)"]
	\\
		F(c'') \ar[dr,"\gamma"'] \ar[rr,"F(g)"{near start}] && F(c') \ar[dr,"\alpha"]
	\\
						       &G(c'') \ar[rr,"G(g)"]\ar[from = uu,"G(g^*f)"{near end}, crossing over]   && G(c').
	\end{tikzcd} \]
	The arrow \( \xi \) is obtained by the universal property of \( G\big((f^*(c'')\big) \) applied to the arrows \( G(f) \circ \beta \circ F(f^*g) \) and \( G(g) \circ \gamma \circ F(g^*f) \), which are equal, as a straightforward calculation shows.
	A pullback of \( f\) along \( g\)---in $\Dialg(F,G)$---is given by the vertical arrows on the left face of the above cube, which are in \( \Dialg(F,G;C)(\gamma) \), since \( g^*f \in C(c'') \).
    Thus we obtain the arrow $\xi$ that makes the following diagram commutative
    \[ \begin{tikzcd}
&&F\big(f^*(c'')\big) \ar[dr,"\xi",dotted]  \ar[dd,"F(g^*f)"] \ar[ddll]
	\\
											  &&&G\big((f^*(c'')\big) \ar[ddll]
	\\
		F(c''')  \ar[dr,"\delta"]  &&F(c'') \ar[dr,"\gamma"'] \ar[ll,"F(h)"] 
	\\
						      & G(c''')& &G(c''),\ar[ll,"G(h)"']\ar[from = uu,"G(g^*f)"{near end}, crossing over] 
	\end{tikzcd} \]
    and the composition $h \circ g^*f$ is the required composition, which lies in \(\Dialg(F,G;C)(c''',\delta)\).
	Clearly, \( X \) is a computability transfer.
	To show the universal property, let a category with base of computability \( \CB E \) together with a computability transfer \( H \colon \CB E \to \CB C \) and a natural transformation \( \beta \colon F \circ H \To G \circ H \).
	The functor \( J \colon \C E \to \Dialg(F,G) \) takes \( e \in E \) to \( \big(H(e),\beta_e\big)\) and \( h \colon e \to e' \) to $H(h)$.
It is immediate to show that this is the only functor, such that \( X \circ J = H \).	
The 2-dimensional part is shown as follows. 
Let computability transfers \( H,H' \colon \CB E \to \CB C \), natural transformations \( \beta \colon F \circ H \To G \circ H , \beta' \colon F \circ H' \to G \circ H' \), and a natural transformation \( \xi \colon H \To H' \), such that the following rectangle commutes
\[ \begin{tikzcd}
	F \circ H \ar[r,"\beta",Rightarrow] \ar[d,"F \ast \xi"',Rightarrow] 
	& G \circ H \ar[d,"G \ast \xi",Rightarrow]
	\\
	F \circ H' \ar[r,"\beta'"',Rightarrow] 
	& G \circ H'.
\end{tikzcd} \]
If \( J,J' \) are functors as in the 1-dimensional part of the universal property, then we define \( \gamma \colon J \To J' \) by the rule \( \gamma_e := \xi_e  \). 
That \(\gamma\) is a natural transformation, follows from the commutativity of the following rectangle
\[ \begin{tikzcd}
	F\big(H(e)\big) \ar[r,"\beta_e"] \ar[d,"F(\xi_e)"'] 
	& G\big(H(e)\big) \ar[d,"G(\xi_e)"]
	\\
	F\big(H'(e)\big) \ar[r,"\beta'_e"'] 
	& G\big(H'(e)\big),
\end{tikzcd} \]
which follows from the commutativity of the preceding rectangle.
A straightforward calculation shows that \( I \ast \gamma = \xi \), 
and it is immediate to show that \( \gamma \) is the only natural transformation satisfying this property.
\end{proof}

\begin{lem}\label{lem: equifier}
	If \( F,G \colon \CB C \to \CB D \) are computability transfers
    %as in the preceding lemma 
    and \( \alpha,\beta \colon F\To G \) are natural transformations,
	then the equifier \( \Eq(\alpha,\beta) \) of \( \alpha \) and \( \beta  \) is given by
	the category \( \Eq(\alpha,\beta) \) with objects those objects \( c \in \C C \), such that \( \alpha_c = \beta_c \), and with arrows the arrows in \( \C C \) between those objects \( c \) in \( \C C \).
	The computability $\Equi(C,D)$ base on \( \Eq(\alpha,\beta) \) consists of all arrows in \( C \) that are in \( \Eq(\alpha,\beta) \), 
    that is 
    \[
    \Equi(C,D)(c) = C(c) \text{ for all }c \in \Eq(\alpha,\beta).
    \]
    The computability transfer \( I \colon \Eq(\alpha,\beta) \to \CB C\) is the inclusion.
\end{lem}

\begin{proof} 
	From the definition of \( \Eq(\alpha,\beta) \) we get \( \alpha \ast I = \beta \ast I  \).
	It is immediate to show that \(\Equi(C,D)\) is closed under identities. To show condition (\texttt{Base}\textsubscript{2}) for $\Equi(C,D)$, let arrows \( f \colon c\to c', g \colon c'' \to c' \) and $h \colon c'' \to c'''$ in \( \Eq(\alpha,\beta) \), such that \( f \in  \Equi(C,D)(c'), h \in \Equi(C,D)(c''') \).
	We have to show that \(h \circ  g^*f  \in \Equi(C,D)(c''')\). To this end it suffices to show $f^*(c'')\in \Eq(\alpha,\beta)$.
    For this, let the following cube
	\[ \begin{tikzcd}
		F\big(F^*(c'')\big) \ar[rr,"F(f^*g)"] \ar[dd,"F(g^*f)"'] 		
		\ar[dr,shift right = .25em,"\alpha_{f^*(c'')}"'{near end}, end anchor = {north west}] 
		\ar[dr,shift left = .25em,"\beta_{f^*(c'')}", end anchor = {north west}] 
		&& F(c) \ar[dr,"\alpha_c"] 	\ar[dd,"F(f)"{near end}]
	\\
		& G\big(f^*(c'')\big) \pullback
\ar[rr,"G(f^*g)"{near start},crossing over]
		&& G(c) \ar[dd,"G(f)"]
		\\
		F(c'') \ar[dr,"\alpha_{c''}"']  \ar[rr,"F(g)"{near start}] 
		&& F(c') \ar[dr,"\alpha_{c'}"]
		\\
		&G(c'') \ar[rr,"G(g)"']\ar[from = uu,"G(g^*f)"{near end} ,crossing over] 		&& G(c').
	\end{tikzcd} \]
	Since both \( \alpha_{G(f^*(c''))} \) and \( \beta_{G(f^*(c''))} \) make the cube commutative, we get from the pullback property of \( G\big(f^*(c'')\big) \) that they are equal.	
From the definition of the base on \( \Eq(\alpha,\beta) \) we have that \( I  \) is a computability transfer.
	To verify the universal property of the equifier, let \( \CB E \) together with \( H \colon \CB E \to \CB C \), such that \( \alpha \ast H = \beta \ast H \).
	We define \( J \colon \C E \to \Eq(\alpha,\beta) \) by the rules \( J(e) := H(e) \), \( J(f) = H(f) \). 
	This is possible, as the equality \( \alpha_{H(e)} = \beta_{H(e)} \) holds by hypothesis. Since \( H \) is a computability transfer, \( J\) is also one.
	That \( H \) is unique follows from the definition and the condition \( I \circ J = H \). For the proof of the 2-dimensional part, we work as follows. If \( H,H' \colon \CB E \to \CB C \) and \( \gamma \colon H \To H' \), we then define \( \omega \colon J \To J' \)---where \( J \) and \( J' \) are defined from \( H \) and \( H' \) as above, respectively---by the rule \( \omega_e := \gamma_e \). It is immediate to show that \( I \ast \omega = \gamma \), and that \( \omega \) is the only natural transformation satisfying  this property.
\end{proof}

\begin{thm}\label{thm: plimits}
 \( \CatBaseComp \) has all limits of pie weight.
\end{thm}
 \begin{proof} 
	 It follows immediately from \cite[Theorem~2.2]{powerCharacterizationPieLimits1991}, as \( \CatBaseComp \) has all products, equifiers and inserters by the previous three lemmata.
 \end{proof}
 
\begin{rem}\label{rem: eqs}
	By Theorem~\ref{thm: plimits} \( \CatBaseComp \) has 2-limits of many shapes, including pseudo-pullbacks and powers. 
	However, $\CatBaseComp$ is not complete, as it fails to have arbitrary equalisers. 
    To illustrate the problem that arises in the general case, 
    first assume that---as in the previous cases---the equaliser of two computability transfers $F,G$ can be obtained by taking the equaliser of the underlying functors in the category of categories, and endowing the resulting category with a special base of computability.
    We know that if a computability base in the equaliser category of \( F,G \) existed, then by condition (\texttt{Base}\textsubscript{2}), the following pullback in $\C C$
	\begin{equation} \begin{tikzcd}
		f^*(b) \ar[r,"i^*f"] \ar[d,"f^*i"', hook] \pullback
		& b \ar[d,"i",hook]
		\\
		a \ar[r,"f"'] & c,
	\end{tikzcd} \label{EqCompMod::eq1} \end{equation}
is mapped into \( \C D \) with \( F,G \) respectively, hence we obtain the following diagram in \( \C D \)
	\[ \begin{tikzcd}
	G\big(f^*(b)\big)\pullback \ar[drr,"G(i^*f)"{near start}, to path = {
	(\tikztostart.east) -| (\tikztotarget.north) \tikztonodes}, 
	rounded corners]
	\ar[ddr,"G(f^*i)"{swap, near start}, to path = {
		(\tikztostart.south) |- (\tikztotarget.west) \tikztonodes
	}, rounded corners]
	\\
	& \pullback F\big( f^*(b)\big) \ar[r,"F(i^*f)"] \ar[d,"F(f^*i)"', hook]
	& F(b) \ar[d,"F(i)", hook]
	\\
	& F(a) \ar[r,"f"'] & c,
	\end{tikzcd} \]
	%in \( \C D \) 
    where both the inner and outer square are pullback squares.
	However, we only obtain an isomorphism, not an identity between \( G\big(f^*(b)\big)  \) and \( F\big(f^*(b)\big) \). Hence, this pullback does not lie in the equaliser of \( F \), \( G \) and the base would not satisfy (\texttt{Base}\textsubscript{2}). 
    
    A possible solution to remedy this is to endow the base of computability with choices of pullbacks, that is, to require the stronger condition:
\begin{enumerate}[leftmargin = 4.25em]
	\renewcommand{\labelenumi}{(\texttt{Base}\textsubscript{\theenumi}\textsuperscript{s})}
	\setcounter{enumi}{1}
\item Given \( a,b\in \C C_0 \) and \( i \colon s \to a, j \colon t \to b \), such that \( i \in C(a), j \in C(b) \) and \(  f \colon s \to b \), there is \emph{specific choice} of a pullback \( s \times_b t \) in 
condition (\texttt{Base}\(_2\))
	\[ \begin{tikzcd}
	& s \times_b t \ar[dl,"i \circ f^*j"{swap, near start}, to path = {(\tikztostart.west) -| (\tikztotarget.north) \tikztonodes}, rounded corners] \ar[r,"j^*f"] \pullback \ar[d,"f^*j"',hook] 
	& t \ar[d,"j", hook] 
	\\
		a & s \ar[l,"i",hook] \ar[r,"f"'] & b.
	\end{tikzcd} \]
%	is a pullback square.
\end{enumerate}
The computability transfers would then be required to preserve these specific choices of pullbacks.
%instead of just (\texttt{Base}\textsubscript{2}).
%We make the above remark precise as follows:
Thus the strategy to endow the equaliser in the category of categories with a base of computability to make it the equaliser in $\CatBaseComp$ does not work, in general.
In the remainder of this section we show that the equalizer of arbitrary $F,G$ in $\CatBaseComp$ does not exist, in general.
\end{rem}

%\begin{lem} 
%	Let \( F,G \colon \CB C \to \CB D \) be computability transfers. 
%	If \( I \colon \CB B \to \CB C \) is an equaliser in \( \CatBaseComp \), then \( \C B \cong \Eq(F,G) \) as categories, where the latter is the equaliser of \( F \) and \( G \) in \( \Cat \).
%\end{lem}

%\begin{proof} 
%	Suppose we are given \( I \colon \CB B \to \CB C \) as in the lemma. 
%	Then for an arbitrary functor \( H \colon \C E \to \C C\) we obtain that \( H \) is a computability transfer \( H \colon (\C E, \total) \to \CB C \), and thus we obtain a unique computability transfer \( J_H \colon (\C E, \total) \to \CB B \) such that \( I \circ J_H = H \). 
%	But this \( J_H \) is simply a functor, so \( H \colon \C B \to \C C \) fulfils the universal property of an equaliser in \( \Cat \), hence we obtain an isomorphism of categories \( \C B \cong \Eq(F,G) \).
%\end{proof}
%
\begin{lem}\label{lem: cospan}
	If \( F,G \colon \CB C \to \CB D \) are computability transfers and 
	if \( I \colon \CB B \to \CB C \) is an equaliser in \( \CatBaseComp \), then for any cospan \( m \in C(c), f \colon c'' \to c \), that is
	\[ \begin{tikzcd}
		c'' \ar[r,"f"] & c & c' \ar[l,"m"'] 
	\end{tikzcd} \]
	for an arbitrary object \( c  \) of \( \C C \), such that \( F(m) = G(m), F(f) = G(f) \), there is a unique cospan \( k,l  \) in \( \C B \) with \( I(k) = m, I(l) = f \), such that \( k \) is in the base of computability \( B \).		
\end{lem}

\begin{proof} 
	Let \( \C S \) be the category with objects \( 0,1,2 \) and  non-identity arrows \( \leq_0 \colon 0 \to 2, \leq_1 \colon 1 \to 2 \). 
	We define a base of computability \( S \) on \( \C S \) by setting \( S(0) = \{1_0\}, S(2) = \{\leq_0, 1_2\}, S(1) = \{1_1\}	 \).
	Considering a cospan \( m \colon c' \to c, f \colon c'' \to c \) in \( \C C \) with \( F(m) = G(m), F(f) = G(f) \), we define the computability transfer \( H \colon \CB S \to \CB C \) that sends \( \leq_0 \) to \( m \) and \( \leq_1 \) to \( f \). 
	Obviously \( F \circ H = G \circ H \), hence, there is a unique \( J_H \colon \CB S \to \CB B \), such that \( I \circ J_H = H \). 
%It is then immediate that
Clearly, \( J_H(\leq_0) \) is the unique arrow \( k \) in \( \C B \) with \( I(k) = m \), and \( J_H(\leq_1)\) is the unique arrow \( l \) with \( I(l) = f \). It is immediate to show that \( k \) is in \( B(J_H(2)) \), since \( \leq_0 \in S(2) \) and \( J_H \) is a computability transfer.
\end{proof}

\begin{prop}\label{prp: graphs}
If \( \C C, \C D \) are the categories generated by the graphs
\[
	 \begin{tikzcd}
		 0 \ar[r,"f"] \ar[d,"g"'] 
		 & 1 \ar[d,"h"]
		 \\
		 2 \ar[r,"i"'] & 3
	\end{tikzcd} 
	\quad\hbox{and}\quad
	\begin{tikzcd} 
		0' \ar[dr,"t", shift left = .5ex] \ar[from = dr,shift left = .5ex, "s"]
		\ar[drr,"k"{near start}, to path = {(\tikztostart) -| (\tikztotarget)\tikztonodes}, rounded corners]
		\ar[ddr,"l"'{near start}, to path ={(\tikztostart) |- (\tikztotarget) \tikztonodes}, rounded corners]
		\\
		& 0 \ar[r,"f"] \ar[d,"g"'] 
		 & 1 \ar[d,"h"]
		 \\
		 &2 \ar[r,"i"'] & 3,
	\end{tikzcd}
\]
respectively, we can endow \( \C C \) with a base of computability \( C \) by setting
\[
	C(0) = \{1_0\}, \quad C(1) = \{1_1\}, \quad C(2) = \{1_2,g\}, \quad C(3) = \{1_3, h\},
\]
and \( \C D  \) with a base of computability \( D \) by setting 
\[
	D(0') = \{1_{0'}\}, \quad D(0) = \{1_0\}, \quad D(1) = \{1_1\}, \quad D(2) = \{1_2,g, l\}, \quad D(3) = \{1_3, h\}.
\]
Then the computability transfers \( F,G \) mapping \( 1,2,3 \) in \( \C C \)---and the corresponding arrows between them---to their counterparts in \( \C D \), but 
\[
	F(0) = 0, F(f) = f, F(f) = g, \quad G(0) = 0', G(f) = k, G(g) = l,
\]
have no equaliser in \( \CatBaseComp \).
\end{prop}

\begin{proof} 
First, we show that \( C,D \) are bases of computability. 
%Clearly, the only square in \( \C C \)---whose horizontal or vertical arrows are not identities---is commutative and the only object with arrows to both \( 1 \) and \( 2 \) is \( 0 \), hence, the universal property of a pullback holds trivially. Both squares in \( \C D \) consisting of \( 0,1,2,3 \) and \( 0',1,2,3 \) are pullback squares. For that, we observe that the square \( 0,1,2,3\) commutes, and the only object besides \( 0 \) with arrows to \( 1 \) and \( 2 \) is \( 0' \). The square \( 0',1,2,3 \) commutes and the unique mediating arrow is \( t \), thus, the universal property of a pullback holds.
%An analogue argument shows that \( 0',1,2,3 \) is a pullback square.
That they are closed under identities, is immediate to show. 
We exhibit only the pullback squares where all arrows of the cospan are not idenditites.
For $C$ these pullback squares are 
\[
\begin{tikzcd}
    0 \ar[r,"f"] \ar[d,"g"'] \pullback & 1 \ar[d,"h"]
    \\
    2 \ar[r,"i"'] & 3,
\end{tikzcd}
\quad\hbox{and}\quad
\begin{tikzcd}
    0 \ar[d,"1_0"] \ar[r,"f"] \pullback & 1 \ar[d,"h"]
    \\
    0 \ar[r,"i \circ g"'] & 3.
\end{tikzcd}
\]
In all cases the pulled back arrows lies in $C$, as required.
For $D$ we note that all necessary pullback squares involving $h$ are given by 
\[ 
\begin{tikzcd}
    0 \ar[r,"f"] \ar[d,"g"'] \pullback & 1 \ar[d,"h"]
    \\
    2 \ar[r,"i"'] & 3,
\end{tikzcd}
\quad
\begin{tikzcd}
    0 \ar[d,"1_0"] \ar[r,"f"] \pullback & 1 \ar[d,"h"]
    \\
    0 \ar[r,"i \circ g"'] & 3
\end{tikzcd}
\quad\hbox{and}\quad
\begin{tikzcd}
    0' \ar[d,"1_{0'}"] \ar[r,"k"] \pullback & 1 \ar[d,"h"]
    \\
    0' \ar[r,"i \circ l"'] & 3.
\end{tikzcd}
\]
Those for $g,l$ are
\[
\begin{tikzcd}
    0' \ar[d,"1_{0'}"'] \ar[r,"t"] \pullback & 0 \ar[d,"g"]
    \\
    0' \ar[r,"l"'] & 2
\end{tikzcd}
\quad\hbox{and}\quad
\begin{tikzcd}
    0 \ar[d,"1_{0}"'] \ar[r,"s"] \pullback & 0' \ar[d,"l"]
    \\
    0 \ar[r,"g"'] & 2.
\end{tikzcd}
\]
It is then straightforward to show that \( F,G \) are computability transfers. 
If there is an equaliser \( I \colon \CB E\to \CB C \) of \( F,G \) in \( \CatBaseComp \), then by Lemma~\ref{lem: cospan} there is a cospan \( \begin{tikzcd}[cramped] e'' \ar[r,"k"] & e & e' \ar[l,"k"'] \end{tikzcd} \) in \( \C E\) 
that maps to the cospan \( \begin{tikzcd}[cramped]
	2 \ar[r,"i"] & 3 & 1 \ar[l,"h"']
\end{tikzcd}
 \) and 
 %fulfils
 \( k \in E(e) \).
Hence, \( k \) is in 
%the base 
\( E \), and the following diagram is a pullback
\[ \begin{tikzcd}
e'^*(e'') \ar[r] \ar[d,"l^*k"'] \pullback & e' \ar[d,"k"]
\\
	e'' \ar[r,"l"'] & e.
\end{tikzcd} \]
%exist. 
Since \( I \) is pullback-preserving, we get \( I\big(l^*k\big) = g \), while \( G(g) \neq F(g) \). Consequently, \( I \)
%\( \CB E \) 
cannot be an equaliser.
\end{proof}

\section{\texorpdfstring{$\CatBaseComp$}{CatBaseComp} and Grothendieck fibrations}\label{sec: fibr}
In this section we show that a Grothendieck fibration 
lifts a base of computability in the base category to a base of computability in the total category of the fibration. Furthermore, using Grothendieck fibrations we regain some pullbacks in $\CatBaseComp$, which do not exist, in general, as it is explained in section~\ref{sec: limits}.

\begin{lem}\label{lem: fib}
    Let $p \colon \C E \to \C B$ be a Grothendieck fibration and $f \colon e' \to e, g \colon e'' \to e$ be a cospan in $\C E$.
    If $g$ is $p$-cartesian for $p(g)$ and $e$, and a pullback of $p(f)$ and $p(g)$ exists in $\C C$, then there exists a pullback of $f$ and $g$ in $\C E$.
\end{lem}

\begin{proof}
    Assuming we are given the following pullback 
    \[
    \begin{tikzcd}
        c \ar[r,"h"] \ar[d,"k"'] \pullback & p(e'') \ar[d,"p(g)"]
        \\
        p(e') \ar[r,"p(f)"'] & p(e),
    \end{tikzcd}    
    \]
    we choose a $p$-cartesian lift $m$  for $k$ and $e'$. 
    Using the universal property of $g$ being $p$-cartesian on $f \circ m$, we obtain a unique arrow $n$
    that makes the following square commutative
    \[ 
    \begin{tikzcd}
        e''' \ar[r,"n", dotted] \ar[d,"m"'] & e'' \ar[d,"g"]
        \\
        e' \ar[r,"f"'] & e.
    \end{tikzcd}\]
    To show that this is a pullback square, let $\tilde e$ with $i \colon \tilde e \to e'', j \colon \tilde e \to e'$ be given, such that the obvious square commutes.
    Applying $p$, and then using the universal property of the pullback in $\C C$, we obtain a unique arrow $o$ that makes the following diagram commutative
    \[ \begin{tikzcd}
    p(\tilde e) \ar[drr,"p(i)", to path = {(\tikztostart.east) -| (\tikztotarget.north) \tikztonodes},rounded corners] 
    \ar[ddr,"p(j)"', to path = {(\tikztostart.south) |- (\tikztotarget.west) \tikztonodes},rounded corners] \ar[dr,"o", dotted]
    \\
        &c \ar[r,"h"] \ar[d,"k"'] \pullback & p(e'') \ar[d,"p(g)"]
        \\
        &p(e') \ar[r,"p(f)"'] & p(e).
    \end{tikzcd}    
    \]
    Using the universal property of $m$ being $p$-cartesian for $k$ and $e'$, we obtain an arrow $l$ that makes the left triangle in the following diagram commutative
    \[
    \begin{tikzcd}
        \tilde e \ar[drr,"i", to path = {(\tikztostart.east) -| (\tikztotarget.north) \tikztonodes},rounded corners]
        \ar[ddr,"j"', to path = {(\tikztostart.south) |- (\tikztotarget.west) \tikztonodes},rounded corners] 
        \ar[dr,"l",dotted]
        \\
        &e''' \ar[r,"n"] \ar[d,"m"'] & e'' \ar[d,"g"]
        \\
        &e' \ar[r,"f"'] & e.
    \end{tikzcd}\]
To show that the upper right triangle commutes, we use the fact that $g$ is $p$-cartesian, and the following commutative triangles 
    \[
    \begin{tikzcd}
        \tilde e \ar[d,"n \circ l"'] \ar[dr,"f \circ j"]
        \\
        e'' \ar[r,"g"'] & e 
    \end{tikzcd}
    \quad\text{and}\quad
    \begin{tikzcd}
        \tilde e \ar[d,"i"'] \ar[dr,"f \circ j"]
        \\
        e'' \ar[r,"g"'] & e 
    \end{tikzcd}
    \]
    are mapped to the commutative triangle 
    \[
    \begin{tikzcd}
        p(\tilde e) \ar[d,"p(i)"'] \ar[dr,"p(f \circ j)"]
        \\
        p(e'') \ar[r,"p(g)"'] & p(e),
    \end{tikzcd}
    \]
hence $n \circ l = i$.
The uniqueness of $l$ follows from the fact that $m$ is $p$-cartesian.
\end{proof}

\begin{lem}\label{prp: monos}
	If \( p \colon \C E \to \C B \) is a Grothendieck fibration, 
	\( f \colon b' \to b \) is a mono, and \( g \colon e' \to e \) is \( p \)-cartesian for \( f  \) and \( e \), then \( g \)	 is a mono.
\end{lem}

\begin{proof} 
	If \( k,l \colon e'' \to e' \), such that \( g \circ k = g \circ l \), then we get the following commutative triangles
	\[ \begin{tikzcd}
		e'' \ar[dr,"g \circ k"] \ar[d,"?"'] \\
		e' \ar[r,"g"'] & e
	\end{tikzcd} 
\quad\hbox{ and }\quad
\begin{tikzcd} 
	b'' \ar[dr,"f \circ p(k)"] \ar[d,"p(k)"']
	\\
	b' \ar[r,"f"'] & b.
\end{tikzcd}
\]
As both \( k,l \) for \( ? \) make the left triangle commutative, we get by  uniqueness---from \( g \) being \( p \)-cartesian---that \( k = l \).	
\end{proof}

\begin{prop}\label{prp: FibLiftBase} 
	If \( F \colon \C E \to \C B \) is a Grothendieck fibration and \( B \) is a base of computability in \( \C B \), then \( F^{-1}(B) \), defined, for every \( e \in \C E \), by
	\[
		F^{-1}(B)(e) = \big\{ i \in \Mon(\textminus, e) ~\vert~ i \ \text{is a cartesian lift of }j \in B\big(F(e)\big)\big\},
	\]
 is a base of computability in \( \C E \).
\end{prop}

\begin{proof}
	%Assume we are given \( \C E, \C B, F \) as in the proposition. 
	%First, we show that \( B_F \) contains identities, and is closed under composition and pullbacks. 
    To show that \( F^{-1}(B) \) contains all unit-arrows, let \( e \in \C E_0 \). We show that \( 1_e \) is a cartesian lift of  \( 1_{F(e)} \).
	If \( g \colon d \to e \) in \( \C E \) and \( h \colon F(d) \to F(e) \), such that the following triangle commutes
	\[ \begin{tikzcd}
		F(d) \ar[dr,"F(g)"] \ar[d,"h"']
		\\
		F(e) \ar[r,"1_{F(e)}"'] & F(e),
	\end{tikzcd} \]
	%commutes, 
    then \( h = F(g) \), and the only possible arrow making the triangle in \( \C E \) commutative is \( g \). Clearly, \( F(g) = h \), and 
	% This shows the desired claim as obviously
     \( F(1_e) = 1_{F(e)} \).
	 To show that \( F^{-1}(B) \) fulfills the second condition, let \( i \colon e' \to e \) in \( F^{-1}(B)(e) \), \( f \colon e'' \to e \), and $j \colon e'' \to e'''$ in $F^{-1}(B)(e''')$ be given.
By Lemma~\ref{lem: fib} we obtain the following pullback in $\C E$
\[ 
\begin{tikzcd}
    \tilde e \ar[r] \ar[d,"h"'] \pullback & e' \ar[d,"i"] 
    \\
    e'' \ar[r,"f"'] & e,
\end{tikzcd}\]
as we can form a pullback in $\C C$ from $B$ being a base, and thus (\texttt{Base}\textsubscript{2}), and $p(i) \in B(p(e))$.
By the proof of Lemma~\ref{lem: fib} $h$ is $p$-cartesian for $p(h)$ and $e''$, and as $j$ is $p$-cartesian for $p(j)$ and $e'''$, by definition of $F^{-1}(B)$, and the fact that the composition of cartesian arrows is cartesian, we get that $j \circ h$ is $p$-cartesian for $p(j \circ h)$ and $e'''$.
\end{proof}

%Notice that \( i \in \Mon(-,e) \) follows immediately.
%does not have to be checked, as it is automatic.

\begin{cor}\label{cor: FibLiftCompTrans} 
	If  \( F \colon \C E \to \C B \) is a Grothendieck fibration and \( B \) is a base of computability in \( \C B \), then \( F \) is a computability transfer with respect to the lifted base \( F^{-1}(B) \). 
\end{cor}

\begin{proof} 
	By the definition of \( F^{-1}(B) \) we have that \( F \) maps monos in \( F^{-1}(B) \) to monos in \( B \). To show the second condition, assume we are given a pullback as it could be used for (\texttt{Base}\textsubscript{2}), that is, both morphisms vertical in the following diagram 
    \[
    \begin{tikzcd}
        f^*(e_1) \ar[r,"f^*g"] \ar[d,"g^*f"'] \pullback & e_2 \ar[d,"f"] 
        \\
        e_1 \ar[r,"g"'] & e_3
    \end{tikzcd}
    \]
    are $p$-cartesian.
    To show that this remains a cartesian square under $p$, assume we are given the following square
    \[
    \begin{tikzcd} 
    b \ar[r,"h"] \ar[d,"k"'] & p(e_2) \ar[d,"p(f)"]
    \\
    p(e_1) \ar[r,"p(g)"'] & p(e_3).
    \end{tikzcd}    
    \]
    We choose a $p$-cartesian lift $\hat k$ of $k$ along $e_1$, and then use that $f$ is $p$-cartesian for $p(f)$ and $e_3$ to obtain an arrow $\hat h$ making the following diagram commutative
    \[
    \begin{tikzcd}
\hat e \ar[drr,to path = {(\tikztostart) -| (\tikztotarget) \tikztonodes}, rounded corners, "\hat h"{near start}]
\ar[ddr, to path = {(\tikztostart) |- (\tikztotarget) \tikztonodes}, rounded corners, "\hat k"'{near start}]
    \\
        &f^*(e_1) \ar[r,"f^*g"] \ar[d,"g^*f"'] \pullback & e_2 \ar[d,"f"] 
        \\
        &e_1 \ar[r,"g"'] & e_3.
    \end{tikzcd}
    \]
    As $f^*(e_1)$ is a pullback, we obtain a unique $\ell$, making the obvious triangles commute. 
    We map $\ell$ into $\C B$ using $p$ in order to get the commutative diagram
    \[
    \begin{tikzcd} 
    b \ar[drr,to path = {(\tikztostart) -| (\tikztotarget) \tikztonodes}, rounded corners,"h"] 
    \ar[ddr, to path = {(\tikztostart) |- (\tikztotarget) \tikztonodes}, rounded corners,"k"'] 
    \ar[dr,"p(\ell)"]
    \\
    &p\big(f^*(e_1)\big) \ar[r,"p(f^*g)"] \ar[d,"p(g^*f)"'] & p(e_2) \ar[d,"p(f)"]
    \\
    &p(e_1) \ar[r,"p(g)"'] & p(e_3).
    \end{tikzcd}    
    \]
    To show that this $p(\ell)$ is the only arrow rendering the above diagram commutative, assume another arrow $m$ satsifies this property as well.
    As $g^*f$ is $p$-cartesian for $e_1$ and $p(g^*f)$, we obtain a unique arrow $n$ with $p(n) = m$ making the following triangle commutative
    \[
    \begin{tikzcd}
        \hat e \ar[r,"n"] \ar[dr,"\hat k"'] & f^*(e_1) \ar[d,"g^*f"]
        \\
        & e_1.
    \end{tikzcd}
    \]
    To see that $f^*g \circ n = \hat h$, observe that $\hat h$ is the unqiue arrow such that $p(\hat h) = h$ and $f \circ \hat h = g \circ \hat k$.
    Thus, since
    \begin{align*}
        f \circ f^*g \circ n = g \circ g^*f \circ n = g \circ \hat k
    \end{align*}
    we get $\hat h = f^*g \circ n$ from uniqueness.
    Thus, $n = \ell$, and hence $m = p(n) = p(\ell)$, showing the required uniqueness of $p(\ell)$.
\end{proof}

%Furthermore, this base is the greatest possible one such that \( F \) becomes a computability transfer.

%\begin{prop} 
%	Let \( p \colon \C E \to \C B \) be a Grothendieck fibration, and \( E \) be a base on \( \C E \) and \( B \) be a base on \( \C B \) such that \( F \) becomes a computability transfer \( p \colon \CB E \to \CB B \).
%	Then \( B_p(e) \supseteq E(e) \) for every \( e \in \C E \).
%\end{prop}
%
%\begin{proof} 
%	Assume that there exists \( i \in E(e) \) for some \( e \) such that \( i \not\in B_F(e) \).
%	Then we know that there exists a \( p\)-cartesian lift \( j \) of \( p(i) \) along \( e \).
%	Thus, as \( p(i) = p(j)	 \), we obtain a unique \( k \) such that \( j \circ k = i \).Now as \( p(k) = 1 \) it is immediate that \( k \) is \( p \)-cartesian as well, because if we are given \( \ell \) with the same codomain, we obtain 
%	\[ \begin{tikzcd}
%		\cdot \ar[dr,"\ell"'] \ar[r,dotted]	&\cdot \ar[d,"k"] \ar[dr,"i"]
%	\\
%	&\cdot \ar[r,"j"] & \cdot 
%	\end{tikzcd} \overset{p}{\mapsto}
%\begin{tikzcd} 
%	\cdot \ar[dr] \ar[d,"p(k)"]
%	\\
%	\cdot \ar[r,"p(i)"'] & \cdot 
%\end{tikzcd}
%\]
%	
%\end{proof}

\begin{defi}
    Let $p \colon \C E \to \C B$ be a functor. 
    An arrow $h \colon e \to e'$ is \emph{weakly} $p$-cartesian for $f \colon b \to b'$ and $e'$, if $p(h) = f$ and for any other $h' \colon e'' \to e'$ with $p(h') = p(h) = f$ there exists a unique $k \colon e'' \to e$ such that the following triangle commutes
    \[
    \begin{tikzcd}
        e'' \ar[dr,"h'"] \ar[d,"k"']
        \\
        e \ar[r,"h"'] & e'.
    \end{tikzcd}
    \]
\end{defi} 

The following lemma is straightforward to show.

\begin{lem}\label{lem: weakcart}
    If $p \colon \C E \to \C B$ is a fibration, then the following are equivalent$:$
    \begin{enumerate}
        \item The arrow $h \colon e \to e'$ is $p$-cartesian for $f \colon b \to b'$ and $e'$.
        \item The arrow $h \colon e \to e'$ is weakly $p$-cartesian for $f \colon b \to b'$ and $e'$.
    \end{enumerate}
\end{lem}

\begin{thm}\label{thm: PullbackFibLift}
    If \( p \colon \C E \to \C B\) is a Grothendieck fibration, 
    \( F  \colon \CB C \to \CB B \) is a computability transfer such that the following diagram is a pullback of categories
    %	Then if we are given a pullback square 
	\[ \begin{tikzcd}
		p^{-1}(\C C) \ar[r,"p^*F"] \ar[d,"F^*p"'] \pullback & \C E \ar[d,"p"]
		\\
		\C C \ar[r,"F"'] & \C B,
	\end{tikzcd} \]
	then the following rectangle
	\begin{equation}\begin{tikzcd}
		\big(p^{-1}(\C C), ({F^*p})^{-1}(C)\big) \ar[r,"p^*F"] \ar[d,"F^*p"'] \pullback &\big( \C E,p^{-1}(B)\big)\ar[d,"p"]
		\\
		\CB C \ar[r,"F"'] & \CB B
	\end{tikzcd} \label{PullbackThmSq} 
    \end{equation}
	is a pullback in \( \CatBaseComp \).
\end{thm}

\begin{proof} 
	%We show the universal property of the pullback.
    We first show that the rectangle \eqref{PullbackThmSq} is in $\CatBaseComp$.
    For this, we show that $p^*F$ maps $F^*p$-cartesian lifts of arrows in $C$ to $p$-cartesian lifts of arrows in $B$.
    To aid readability, we work with the canonical description of the pullback in $\Cat$ with objects pairs $(c,e)$ such that $F(c) = p(e)$, and arrows pairs $(f,g)$, such that $F(f) = p(g)$.
    Assume $(f,g) \colon (c,e) \to (c',e')$ is $F^*p$-cartesian for $f \colon c \to c'$ and $(c',e')$, where $f \in C(c')$.
    By Lemma~\ref{lem: weakcart} it suffices to show that $p^*F(f,g) = g$ is weakly $p$-cartesian for $F(i)$ and $p^*F(c',e') = e'$.
    Let $h \colon e'' \to e'$ be another arrow with $p(h) = F(i)$.
    Forming the arrow $(f,h) \colon (c,e'') \to (c',e')$ then yields a unique arrow $(m,n) \colon (c,e'') \to (c,e)$---as $F^*p(f,h) = f$ and $(f,g)$ is $F^*p$-cartesian for $f$ and $(c',e')$---making the following triangle in $p^{-1}(\C C)$ commutative
    \[ \begin{tikzcd}
    (c,e'') \ar[d,"{(m,n)}"'] \ar[dr,"{(f,h)}"]
    \\
    (c,e) \ar[r,"{(f,g)}"'] & (c',e').
    \end{tikzcd}\]
    Thus, the following triangle in $\C E$ is commutative
      \[ \begin{tikzcd}
    e'' \ar[d,"n"'] \ar[dr,"h"]
    \\
    e \ar[r,"g"'] & e'.
    \end{tikzcd}\]   
    To see that $n$ is unique, observe that any other $n'$ making the above triangle commutative would necessarily make the following triangle
     \[ \begin{tikzcd}
    (c,e'') \ar[d,"{(m,n')}"'] \ar[dr,"{(f,h)}"]
    \\
    (c,e) \ar[r,"{(f,g)}"'] & (c',e')
    \end{tikzcd}\]
    commutative, which contradicts the uniqueness of $(m,n)$.
    That $p^*F$ preserves the pullbacks used in (\texttt{Base}\textsubscript{2}) is shown with a straightforward calculation.
    
	If the following diagram commutes
    %we are given another commutative square
	\[ \begin{tikzcd}
		\CB G \ar[r,"H"] \ar[d,"G"'] &\CB E \ar[d,"p"]
		\\
		\CB C \ar[r,"F"'] & \CB B,
	\end{tikzcd} \]
	then we obtain---in \( \Cat \)---a unique functor \( \langle G,H \rangle \colon \C G \to p^{-1}(\C C) \) making the following diagrams
	\[ \begin{tikzcd}
	& \C G \ar[dr,"H"] \ar[dl,"G"'] \ar[d,"{\langle G,H \rangle}"{description}]
	\\
		\C C & p^{-1}(\C C) \ar[r,"p^*F"'] \ar[l,"F^*p"] & \C E
	\end{tikzcd} \]
	commutative.
	We show that \( \langle G,H\rangle \) is a computability transfer.
	For this, let \( i \in G(x) \) be given. 
	It suffices to show that \( \langle G,H \rangle(i) \) is \( F^*p \)-cartesian for \( G(i) \) and \(\langle G,H \rangle(x)\).
	If \( \ell  \) has codomain \( \langle G,H \rangle(x) \) and \( k \) is an arrow, such that \( G(i) \circ k = F^*p(\ell) \), we can apply \( F \) to obtain the same triangle in \( \C B \), that is, we have
	\[
		\begin{tikzcd} 
			\textminus \ar[dr,"\ell"]
			\\
		\textminus \ar[r,"{\langle G,H \rangle}i"'] &{\langle G,H \rangle}(x) 
		\end{tikzcd}
		\hbox{ and }
		\begin{tikzcd} 
			\textminus \ar[dr,"F^*p(\ell)"]
			\ar[d,"k"']
			\\ 
			\textminus \ar[r,"G(i)"'] & G(x)
		\end{tikzcd}
		\hbox{ and }
		\begin{tikzcd} 
			\textminus \ar[dr,"{(F \circ F^*p)(\ell)}"]
			\ar[d,"F(k)"']
			\\
			\textminus \ar[r,"{(F \circ G)(i)}"'] & F\big(G(x)\big)
		\end{tikzcd}
	\]
	in \( p^{-1}(\C C), \C C \) and \( \C B \), respectively. 
	Since \( p \) is a fibration, we can lift the last triangle to a triangle
	\[ \begin{tikzcd}
		\textminus \ar[dr,"p^*F(\ell)"]
		\ar[d,"f"']
		\\
		\textminus \ar[r,"H(i)"']& H(x),
	\end{tikzcd} \]
	since \( H \) is a computability transfer and, thus, \( H(i) \) is cartesian for \( F(G(i)) = p(H(i)) \).
	Next, we observe that \( \langle G,H\rangle i \) is the only arrow in \( p^{-1}(\C C) \) that is mapped to \( H(i) \) by \( p^*F \) and to \( G(i) \) by \( F^*p \). Otherwise, there would be no  unique fill for 
	\[ \begin{tikzcd}
	& \Delta(1)  \ar[dr,"\mapsto H(i)"] \ar[dl,"\mapsto G(i)"'] \ar[d,"?", dotted, description]
	\\
		\C C & p^{-1}(\C C) \ar[r,"p^*F"'] \ar[l,"F^*p"] & \C E.
	\end{tikzcd} \]
    Similarly, 
	%A similar argument shows that
    \( \ell \) is uniquely determined by its images under \( p^*F\) and \(F^*p \).
	If we consider \( \Delta(2) \) and the arrows 
	\[ 
		\begin{tikzcd} 
			\textminus \ar[dr,"F^*p(\ell)"]
			\ar[d,"k"']
			\\ 
			\textminus \ar[r,"G(i)"'] & G(x)
		\end{tikzcd}
		\mapsfrom 
		\begin{tikzcd} 
			0 \ar[d] \ar[dr] 
			\\
			1 \ar[r] & 2
		\end{tikzcd}
		\mapsto 
		\begin{tikzcd}
			\textminus \ar[dr,"p^*F(\ell)"]
			\ar[d,"f"']
			\\
			\textminus \ar[r,"H(i)"']& H(x),
		\end{tikzcd}
	\]
	 then from the pullback property we get the required commutative triangle
	 	\[
			\begin{tikzcd} 
				\textminus \ar[dr,"\ell"] \ar[d,"g"]
			\\
			\textminus \ar[r,"{\langle G,H \rangle}(i)"'] &{\langle G,H \rangle}(x) 
		\end{tikzcd}
	\]
	(where \( g \) is unique from the uniqueness of the pullback), showing that \( \langle G,H \rangle i \) is indeed \( F^*p \)-cartesian for \( G(i) \).
\end{proof}

\section{Transporting bases of computability}\label{sec: transport}
In this section we define various new bases of computability from given ones. Namely, we show that a pullback-preserving fibration $F \colon \mathscr{E \to B}$ maps a base of computability $E$ in $\mathscr{E}$ to a base of computability $F(E)$ in $\mathscr{B}$. Moreover, we describe the base of computability in the comma-category $F/G$ and the exponential in $\CatBaseComp$.

\begin{prop}\label{prp: basetobase}
Let $\mathscr{E}$ be a category with base of computability $E$. 
If $F \colon \mathscr{E \to B}$ is a pullback-preserving fibration, then  $F(E)$ is a base of computability in $\C{B}$, where
\[ F(E)(d) = \{F(i) ~\vert~ i \in E(c) \wedge F(c) = d\} \cup \{ 1_d\}, \]
and $F$ is a computability transfer $\CB E \to (\C C,F(E))$.
\end{prop}

\begin{proof}
Condition (\texttt{Base}\textsubscript{1}) follows immediately, since $1_e  \in F(B)(e)$, for every $e \in \C E$.
To show condition (\texttt{Base}\textsubscript{2}),
let the following arrows $i,j,f$ in $\C{B}$
\[ \begin{tikzcd}
&&d_4 \ar[d,"j",hook]
\\
d_1 & d_2 \ar[l,"i",hook] \ar[r,"f"'] & d_3,
\end{tikzcd} \]
%be given 
where $i \in F(B)(d_1)$ and $j \in F(B)(d_3)$.
We distinguish two cases. If 
%\begin{itemize}[leftmargin = 1.5em]
	%\item 
$j = {1}_{d_3}$, then $f^* 1_{d_3} = 1_{d_3}$, and thus, 
	\( i \circ f^* 1_{d_3} = i \circ 1_{d_3} = i \in F(B)(d_1) \)	by definition of $i$.
If $j = F(k)$, for some $k \colon c_4 \to c_3$, where $k \in B(c_4)$, we obtain an $F$-cartesian lift $f'\colon c_2 \to c_3$ of $f$, since $F(c_3) = d_3$. Let the following pullback
%	We can thus compute the pullback 
	\[ \begin{tikzcd}
	k^*(c_2) \ar[r,"k^*f'"] \ar[d,"{f'}^*k"'] & c_4 \ar[d,"k"] 
	\\
	c_2 \ar[r,"f'"'] & c_3.
	\end{tikzcd} \]
Since $F$ is pullback-preserving, we get $F(k^*(c_2)) \cong j^*(d_2)$. Furthermore, as $F(c_2) = d_2$, we have, from the definition of the base,  that either $i$ is always given as $h \colon c_2 \to c_1$, such that $F(h) = i$, since in the case that $i$ is the identity we simply have that $h = 1_{c_2}$. 
Hence, $h \circ {f'}^*k \in B(c_1)$, and thus, we get the equality $F(h \circ {f'}^*k) = F(h) \circ F( {f'}^*k) \in F(E)(d_1)$, as required.
It is immediate to show that $F$ is a computability transfer.
\end{proof}
%\begin{rem}[]
%	Observe that we did not need that the cartesian square transported by \( F \) was still a cartesian square, so our fibration \( F \) need not be pullback-preserving!
%\end{rem}
\begin{lem}\label{lem: fibr1}
If $F\colon \C{E \to B}, G \colon \C{Z \to E}$ are pullback-preserving fibrations, and if $Z$ is a base of computability in $\C Z$, then the following hold:\\[1mm]
\normalfont 
(i) \itshape $1_{\C Z}(Z) = Z$.\\[1mm]
\normalfont 
(ii) \itshape $({G \circ F})(Z) = G\big(F(Z)\big)$.
\end{lem}

\begin{proof}
Case (i) follows immediately. For the proof of case (ii),
we observe that if $i \in ({G \circ F})(Z)(e)$, for some $e \in \C E$, then, either $i = 1_a$, and thus, $i \in (G\big(F(Z)\big)(e)$, or $i = G(F(k))$, for some $k \in B(c)$, such that $G\big(F(c)\big) = e$. Hence, $F(k) \in F(Z)\big(F(c)\big)$, and thus, $G\big(F(k)\big) = i \in (G \circ F)(Z)(e)$. 
The converse inclusion is shown similarly.
\end{proof}

\begin{defi}\label{def: comma}
    If $\C{A,B,C}$ are categories and $F \colon \C{A \to B}, G \colon \C{C \to B}$ are functors, then the \emph{comma-category} $F/G$ has object triplets $(a,h,c)$ of objects $a \in \C A, c \in \C C$ and an arrow $h \colon F(a) \to G(c)$. Its arrows
are pairs $(f,g) \colon (a,h,c) \to (a',h',c')$ of arrows $f \colon F(a) \to F(a')$ and $ g \colon G(c) \to G(c')$, such that the following rectangle commutes
        \[ \begin{tikzcd}
            F(a) \ar[r,"f"] \ar[d,"h"'] & F(a') \ar[d,"h'"]
            \\
            G(c) \ar[r,"g"'] & G(c').
        \end{tikzcd}\]
        Moreover, $(i,j) \circ (f,g) = (i \circ f, j \circ g)$.
\end{defi}

\begin{lem}\label{lem: commapullbacks}
     Let $\C{A,B,C}$ be categories. If the pullbacks of the following cospans exist,
     \[
     \begin{tikzcd}
         a_1 \ar[r,"f"] & a_2 & a_3 \ar[l,"i"']
     \end{tikzcd}
     \quad\hbox{and}\quad
     \begin{tikzcd}
         c_1 \ar[r,"g"] & c_2 & c_3, \ar[l,"j"']
     \end{tikzcd}
     \]
     and $F \colon \C{A \to B}, G \colon \C{C \to B}$ preserve these pullbacks, then a pullback of 
     \begin{equation} \begin{tikzcd}
        (a_1,h_1,c_1) \ar[r,"{(f,g)}"] & (a_2,h_2,c_3) & (a_3,h_3,c_3). \ar[l,"{(i,j)}"'] 
    \end{tikzcd}\label{CommaPullbackEq::1} \end{equation}
    exists in $F/G$.
\end{lem}

\begin{proof}
\newcommand{\ahc}[1]{(a_#1,h_#1,c_#1)}
    The cospan~\eqref{CommaPullbackEq::1} is induced by the cospans in $\C A, \C C$, which are transported to $\C B$ as follows: 
    \[ \begin{tikzcd}
        a_1 \ar[d,"f"'] \ar[r,""{name = V1}, phantom] & F(a_1)\ar[d,"F(f)"'] \ar[r,"h_1"] & G(c_1)\ar[d,"G(i)"] & c_1 \ar[d,"i"] 
        \ar[l,""{name = U1}, phantom]
        \\
        a_2  \ar[r,""{name = V2}, phantom] & F(a_2) \ar[r,"h_2"] & G(c_2) & c_2 
        \ar[l,""{name = U2}, phantom]
        \\
        a_3 \ar[u,"g"]  \ar[r,""'{name = V3}, phantom] & F(a_3) \ar[r,"h_3"']\ar[u,"F(g)"]  & G(c_3)\ar[u,"G(j)"'] & c_3 \ar[u,"j"'] \ar[l,""'{name = U3}, phantom] 
        \\
        \mathscr{A}& \phantom{0} \ar[r,"\mathscr{B}"{font = \normalsize},phantom] &\phantom{0} & \mathscr{C}.
        \arrow[from = V1, to = V3, end anchor = {[yshift = -1.25cm]},  no head, dashed] 
        \arrow[from = U1, to = U3, end anchor = {[yshift = -1.25cm]}, no head, dashed]
    \end{tikzcd}\]
    If we take the pullbacks of both the cospan in $\C A$ and in $\C C$, then these become the pullbacks of the two cospans in $\C B$, since both $F$ and $G$ preserve those pullbacks. This yields
    \begin{equation}\begin{tikzcd}
        & F(a_1)\ar[d,"F(f)"'] \ar[r,"h_1"] & G(c_1)\ar[d,"G(i)"] &
        \\
       F(i^*(a_1)) \ar[ur,"F(g^*f)"] \ar[dr,"F(f^*g)"'] & F(a_2) \ar[r,"h_2"] & G(c_2) & G(j^*(c_1)),\ar[ul,"G(i^*j)"'] \ar[dl,"G(j^*i)"]
        \\
       & F(a_3) \ar[r,"h_3"']\ar[u,"F(g)"]  & G(c_3)\ar[u,"G(j)"'] & 
    \end{tikzcd} \label{CommaPullbackEq::2}\end{equation}
    where all squares commute. 
    This in turn entails
    \begin{align*}
        G(i) \circ h_1 \circ F(f^*g) &= h_2 \circ F(f) \circ F(F^*g) = h_2 \circ F(g) \circ F(g^*f)
        \\
        &= G(j) \circ h_3 \circ F(g^*f). 
    \end{align*}
    From the universal property of pullbacks there is $\tilde h \colon F(i^*(a_1)) \to G(j^*(c_1))$, such that 
    \begin{equation} G(j^*i) \circ \tilde{h} = h_3 \circ F(f^*g) \quad\hbox{and}\quad G(i^*j) \circ \tilde h = h_1 \circ F(g^*f).\label{CommaPullbackEq::3} \end{equation}
    Hence, 
    \((f^*g, j^*i) \colon (i^*(a_1), \tilde h, j^*(c_1))\to \ahc{3}\)
    and \((g^*f, i^*j) \colon \ahc{1} \to (i^*(a_1), \tilde h, j^*(c_1))\) are arrows in $F/G$.
    It remains to show that $(i^*(a_1) , \tilde h , j^*(c_1))$ satisfies the universal property of a pullback.
    Let $(a',h',c') \in F/G$, arrows $(n_1,m_1) \colon (a',h',c') \to \ahc{1}$ and $(n_3,m_3) \colon (a',h',c') \to \ahc{3}$, such that the following rectangle commutes
   \[\begin{tikzcd}
        (a',h',c') \ar[r,"{(n_3,m_3)}"] \ar[d,"{(n_1,m_1)}"'] & \ahc{3} \ar[d,"{(i,j)}"] 
        \\
        \ahc{1} \ar[r,"{(f,g)}"'] & \ahc{2}.
    \end{tikzcd} \]
   This in turn decomposes to the respective commutative squares in $\C A, \C C$, hence, we can use the pullback properties of $i^*(a_1)$ and $j^*(c_1)$ to obtain arrows $n',m'$ making the following diagrams commutative
    \[ 
    \begin{tikzcd}
        &a' \ar[dr,"n_3"]  \ar[d,"n'"] \ar[dl,"n_1"'] 
        \\
        a_1 & i^*(a_1)\ar[l,"f^*g"] \ar[r,"g^*f"']& a_3
    \end{tikzcd}
    \quad
    \begin{tikzcd}
        & c' \ar[dr,"m_3"] \ar[d,"m'"] \ar[dl,"m_1"']
        \\
        c_1 & j^*(c_1)  \ar[l,"i^*j"] \ar[r,"j^*i"'] & c_3.
    \end{tikzcd}\]
   % commute.
    It remains to show that the following rectangle commutes 
    \[ \begin{tikzcd}
        F(a') \ar[d,"F(n')"'] \ar[r,"h'"] & G(c') \ar[d,"G(m')"]
        \\
        F(i^*(a_1) \ar[r,"\tilde h"'] & G(j^*(c_1).
    \end{tikzcd}\]
        We pack all arrows considered so far into the following  diagram
    
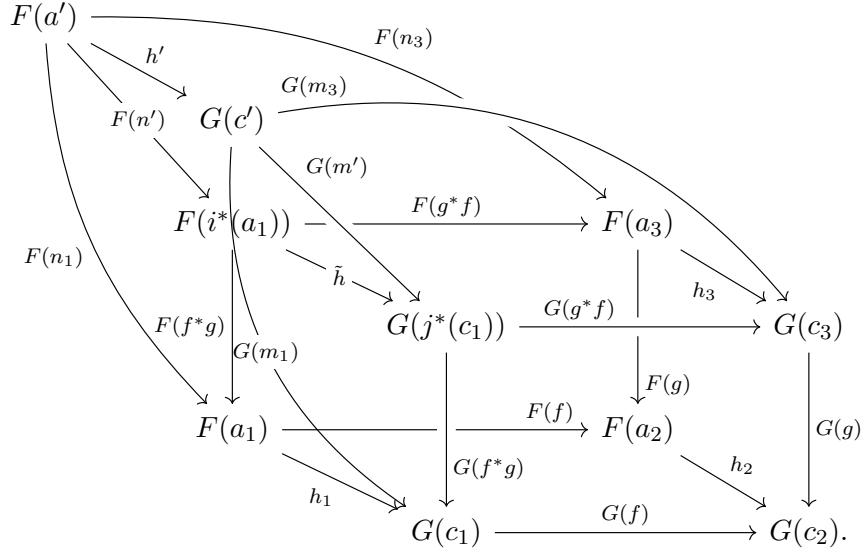
\[ \begin{tikzcd}
        F(a') \ar[dr,"h'"] \ar[ddrrr,"F(n_3)", bend left = 20] \ar[ddddr, bend right = 20, "F(n_1)"']
        \ar[ddr,"F(n')" description]
        \\
        &G(c')  \ar[ddddr, bend right = 30, "G(m_1)"{description}]\ar[ddrrr, "G(m_3)"{very near start}, bend left = 30, crossing over] 
        \\
        & F(i^*(a_1)) \ar[rr,"F(g^*f)"] \ar[dd,"F(f^*g)"'] && F(a_3) \ar[dd, "F(g)"{very near end}] \ar[dr,"h_3"']
        \\
        &&G(j^*(c_1)) \ar[from = ul,"\tilde h"{description}, crossing over]  \ar[from = uul, crossing over,"G(m')"{near start}] \ar[rr,"G(g^*f)"{near start},crossing over] && G(c_3) \ar[dd,"G(g)"]
        \\
        & F(a_1) \ar[dr,"h_1"'] \ar[rr,"F(f)"{very near end}] && F(a_2) \ar[dr,"h_2"]
        \\
        &&G(c_1)\ar[from = uu,"G(f^*g)",crossing over, near end] \ar[rr,"G(f)"] && G(c_2).
    \end{tikzcd}\]
    In the preceding diagram, all rectangles commute (except the one we want to prove), as well as the triangles of the pullbacks. From that we get
   % This allows us to compute that 
    \begin{align*}
        G(f) \circ G(m_1) \circ h' &= G (f) \circ G(f^*g) \circ G(m') \circ h' = G(g) \circ G(g^*f) \circ G(m')  \circ h'
        \\
        &= G(g) \circ G(m_3) \circ h'.
    \end{align*}
    Thus, by the pullback property there is an arrow $\check h \colon F(a') \to G(j^*(c_1))$, such that $G(g^*f) \circ \check h = G(m_3) \circ h$ and $G(f^*g) \circ \check h = G(m_1) \circ h'$. 
    Clearly, both $\check h = G(m') \circ h'$ and $\check h = \tilde h \circ F(n')$ satisfy these conditions. For $G(m') \circ h'$, this is immediate, and for $\tilde h \circ F(n')$, we have that 
    \begin{align*}
        G(g^*f) \circ \tilde h \circ F(n') &= h_3 \circ F(g^*f) \circ F(n') = h_3 \circ F(n_3) 
        = G(m_3) \circ h'
    \end{align*}
    and 
    \begin{align*}
        G(f^*g) \circ \tilde h \circ F(n') &= h_1 \circ F(f^*g) \circ F(n') = h_1 \circ F(n_1) 
        = G(m_1) \circ h'.
    \end{align*}
    By uniqueness, we get the required equality.
\end{proof}

\begin{prop}\label{prp: commabase}
    Let $F \colon \C{A \to B}, G \colon \C{C \to B}$ be pullback-preser\-{}ving functors. 
    Let $A,B,C$ be bases of computability in $\C{A,B,C}$, respectively. 
    Then $B_/$ is a base of computability in $F / G$, where 
    \[ B_/(a,h,c) = \big\{ (i,j) ~\vert~ i \in A(a) \,\&\,F(i) \in B(F(a)) \,\&\, j \in C(c) \,\&\, G(j) \in B(F(c))\big\}. \]
\end{prop}

\begin{proof}Since the identity on $(a,h,c)$ is $(1_a, 1_c)$ and $1_a \in A(a), F(1_a) = 1_{S(a)} \in B\big( F( a)\big)$ (we work similarly for $c$),  condition (\texttt{Base}\textsubscript{1}) is met.
    To show condition (\texttt{Base}\textsubscript{2}), we assume we are given objects and arrows as in the following diagram 
    \[ \begin{tikzcd}
        &&(a_1,h_1,c_1) \ar[d,"{(i_1,j_1)}"] \\
        (a_2,h_2,c_2) & (a_3, h_3, c_3) \ar[l,"{(i_2,j_2)}"] \ar[r,"{(f,g)}"']& (a_4,h_4,c_4),
    \end{tikzcd}\]
    where we assume that $(i_1,j_1), (i_2,j_2)$ are in the computability base $B_/$ and $(f,g)$ is arbitrary. 
    Since by Lemma~\ref{lem: commapullbacks} $F / G$ has pullbacks, let the following pullback 
    \[ \begin{tikzcd}
        & (i_1^*(a_3), \tilde h, j_1^*(c_3)) \ar[d,"{(f^*i_1, g^*j_1)}"'] \ar[r,"{(i_1^*f, j_1^*g)}"]&(a_1,h_1,c_1) \ar[d,"{(i_1,j_1)}"] \\
        (a_2,h_2,c_2) & (a_3, h_3, c_3) \ar[l,"{(i_2,j_2)}"] \ar[r,"{(f,g)}"']& (a_4,h_4,c_4).
    \end{tikzcd}\]
    As $(i_2, j_2) \circ (f^*i_1, g^*j_1) = (i_2 \circ f^*i_1, j_2 \circ g^*j_1)$ and both of these arrows are in $A, C$, respectively, it remains to show that $F(i_2 \circ f^*i_1) \in B\big( F( a_2)\big) $ and $G(j_2 \circ g^*j_1) \in B\big(G(c)\big)$. By definition of $B_/$ we get $F(i) ,F(f) \in B\big(F(a_4)\big)$, $F(i_2) \in B\big(F(a_2)\big)$. 
    Hence, $F(f^*i_1) = F(f)^*F(i_1) \in B\big(F(a_3)\big)$, since $F$ is pullback-preserving. 
    Thus, $F(f^*i_1) \circ F(i_2) = F(f^*i_1 \circ i_2) \in B\big(F(a_2)\big)$. 
    The same argument works for $G$ and the second components of the arrows. 
\end{proof}

\begin{exas}\label{ex: commas}
(1) Given a category $\C C$ with a base of computability $B$, and if $c$ is an object of $\C C$, we can endow the slice category $\C C/c$ with a base of computability $B{/c}$, defined by 
        \[ (B{/c})(f\colon b \to c) = \{ i \colon a \to b ~\vert~ i \in B(b)\}.  \]
        This follows from the fact that the slice category is the comma category of the cospan 
        \[ \begin{tikzcd}
            \C C \ar[r,"1_{\C C}"] & \C C & \bbone, \ar[l,"\iota_c"'] 
        \end{tikzcd}\]
        where $\iota_c$ is the functor that maps the only object $0$ of $\bbone$ to $c$ and its identity $1_{0}$ to $1_c$.\\
(2) Similarly, we can endow the coslice category ${c/}\C C$ with the computability base ${c/}B$
        \[ ({c/}B)(f \colon c \to b) = \{ i \colon a \to b~\vert~ i \in B(b)\}. \]
        This follows from the fact that the coslice category is the comma category of the following cospan 
        \[ \begin{tikzcd}
         \bbone \ar[r,"\iota_c"] & \C C & \C C. \ar[l,"1_{\C C}"'] 
         \end{tikzcd}
        \]
(3) If $B$ is a base of computability on a category $\C C$ we can equip the arrow category $\C C^\to$ with the base of computability
        $B^\to$ defined by 
        \[ B^\to(f \colon c \to d) = \{ (i,j) \colon (c',d') \to (c,d) ~\vert~ i \in B(c) \wedge j \in B(d)\}. \]
        This follows from the fact that the arrow category is the comma category of the following cospan
        \[ \begin{tikzcd}
            \C C\ar[r,"1_{\C C}"] & \C C & \C C.\ar[l,"1_{\C C}"'] 
        \end{tikzcd}\]
\end{exas}

\begin{defi}
    Given two categories with bases of computability, $\CB C, \CB D$, the \emph{computability transfer category} $[\C C, \C D]_{(C,D)}$ has 
    objects computability transfers $F \colon \CB C \to \CB D $, and arrows natural transformations $\eta \colon F \To G$ such that all naturality squares 
    \[
\begin{tikzcd}
    F(c) \ar[r,"F(f)"] \ar[d,"\eta_c"'] \pullback & F(c') \ar[d,"\eta_{c'}"]
    \\
    G(c) \ar[r,"G(f)"'] & G(c')
\end{tikzcd}
    \]
    are pullback squares.
\end{defi}

\begin{lem}\label{lem::exponent1}
    In the  computability transfer category $[\C C, \C D]_{(C,D)}$ a pullback of $\eta$ along $\mu$ exists, if for every $c \in \C C$ a pullback of $\eta_c$ along $\mu_c$ exists.
\end{lem}

\begin{proof}
    Suppose we are given a cospan $\begin{tikzcd} F_2 \ar[r,"\mu",Rightarrow] & F_1 & F_3 \ar[l,"\eta"',Rightarrow]\end{tikzcd}$ of computabiltiy transfers and natural transformations. 
    For every $c \in \C C$ we consider the following pullback, which exists by assumption,
    \[ 
    \begin{tikzcd}
    \eta_c^*(F_2(c)) \ar[r] \ar[d] \pullback & F_3(c) \ar[d,"\eta_c"]
    \\
    F_2(c) \ar[r,"\mu_c"'] & F_1(c).
    \end{tikzcd} 
    \]
    We define $\eta^*(F_2)$ via $\big(\eta^*(F_2)\big)(c) := \eta_c^{-1}\big(F_2(c)\big)$. If $f \colon c \to c'$ let the following cube 
    \begin{equation}\begin{tikzcd}
        \eta_c^{*}(F_2(c)) \ar[rr] \ar[dr,dotted,"q"]\ar[dd] \pullback && F_3(c) \ar[dd,"\eta_c"{near end}] \ar[dr,"F_3(f)"]
   \\
    &\eta_{c'}^{*}(F_2(c')) \ar[rr] \ar[dd] \pullback && F_3(c') \ar[dd,"\eta_{c'}"]
     \\
    F_2(c) \ar[rr,"\mu_c"'{near end}] \ar[dr,"F_2(f)"'] && F_1(c) \ar[dr,"F_1(f)"]
    \\
    &F_2(c') \ar[rr,"\mu_{c'}"'] && F_1(c'),
    \end{tikzcd}\label{Exponent::diag1}
    \end{equation}
    where $q$ is given by the universal property of the pullback $\eta_{c'}^{*}\big(F_2(c')\big)$.
    We set $\big(\eta^*(F_2)\big)(f) := q$.
    This construction preserves identities and compositions of arrows by definition, hence $\eta^*(F_2)$ is a functor.
    It remains to check that $\eta^*(F_2)$ preserves the pullbacks used in condition (\texttt{Base}\textsubscript{2}) and $\eta^*(F_2)(f) \in D\big(\eta^*(F_2)(c')\big)$ if $f \in C(c')$.
    Note that if the following pullback is given in $\C C$---with $i \in C(c_1)$
    \[\begin{tikzcd}
        c_4 \ar[r,"p_1"] \ar[d,"p_2"'] &c_2 \ar[d,"i"]
        \\
        c_3 \ar[r,"f"'] & c_1,
    \end{tikzcd}\]
    then the following squares are pullbacks, for every $\triangle \in \{F_1,F_2,F_3\}$. 
    \[ \begin{tikzcd}
        \triangle(c_4) \ar[r,"\triangle(p_1)"] \ar[d,"\triangle(p_2)"'] & \triangle(c_2) \ar[d,"\triangle(i)"]
        \\
        \triangle(c_3) \ar[r,"\triangle(f)"'] & \triangle(c_1).
    \end{tikzcd}\]
    It requires a straightforward calculation to show that the following rectangle is a pullback
     \[ \begin{tikzcd}
        \big(\eta^*(F_2)\big)(c_4) \ar[r,"\big(\eta^*(F_2)\big)(p_1)"] \ar[d,"\big(\eta^*(F_2)\big)(p_2)"'] &[3em] \big(\eta^*(F_2)\big)(c_2) \ar[d,"\big(\eta^*(F_2)\big)(f)"]
        \\
        \big(\eta^*(F_2)\big)(c_3) \ar[r,"\big(\eta^*(F_2)\big)(g)"'] & \big(\eta^*(F_2)\big)(c_1).
    \end{tikzcd}\]
    If $f \in C(c')$, then in the cube~\eqref{Exponent::diag1} the back face and the right face can be pasted to form a composite pullback, thus, using the commutativity of all squares, the front face pasted with the left face is a pullback as well.
    Since the front face is a pullback, by the pasting property of pullbacks the left side is a pullback as well.
    But as $F_2(f) \in D(F_2(c'))$---since $F_2$ is a computability transfer---condition (\texttt{Base}\textsubscript{2}) yields that $q \in D\big(\eta_{c'}^*\big(F_2(c')\big)\big)$, hence, $\eta^*F_2$ is a computability transfer.
That it is a pullback, is straightforward to check.
\end{proof}

\begin{rem}
    If we restrict ourselves to replete bases of computability, the assumption that all naturality squares of the natural transformation involved are cartesian can be dropped, and replaced by the assumption that $\eta_c \in D(F_1(c))$, for every $c \in \C C$.
    This follows from the fact that if we are given arrows $f \colon c \to c', g \colon c' \to c''$, such that, $g$ is a mono and $g \circ f \in C(c'')$, then the following diagram
    \[
    \begin{tikzcd}
        c \ar[r,"1_c"] \ar[d,"f"'] \pullback & c \ar[d,"g \circ f"] 
        \\
        c' \ar[r,"g"'] & c''
    \end{tikzcd}
    \]
     is a pullback square, hence, $f \in C(c')$.
    Applying this fact on the left face of the cube~\eqref{Exponent::diag1}, we get that $q \in D\big(\eta_{c'}^*\big(F_2(c')\big)\big)$, since $\eta_{c'}, \eta_c$ lie in $D$, hence, their pullbacks do so as well, and as $F_2$ is a computability transfer, the composite of the lower and left arrow of the left face lies in $D$.
    Since the right arrow lies in $D$ as well, so does the top arrow $q$.
\end{rem}

\begin{prop}\label{prp: basefunctorcat}
    The computability transfer category $[\C C,\C D]_{(C,D)}$ has a computability base $[\C C,D)]$ defined by
    \[ [\C C, D](F) = \big\{ \eta \colon F' \Rightarrow F ~\vert~ \forall_{c \in \C C}\big(\eta_c \in D\big(F(c)\big)\big)\big\}.\]
\end{prop}

\begin{proof}
   % We show that $[\C C, B]$ defined in this manner constitutes a base of computability. 
    The proof of condition (\texttt{Base}\textsubscript{1}) is immediate, since if $F$ is a functor, then the identity $1_F$ is the natural transformation consisting of the identities $1_{F(c)} \colon F(c) \to F(c)$, for every $c$. 
    But $1_{F(c)}$ is in $D(F(c))$, for every $c \in \C C$, since $D$ is a base of computability, so $1_F$ is in $[\C C, D]$.
    To show condition (\texttt{Base}\textsubscript{2}), let functors and natural transformations as in the following diagram
    \[ 
    \begin{tikzcd}
        && F_4 \ar[d,"\eta",Rightarrow]
        \\
        F_1 & F_2 \ar[l,"\chi",Rightarrow] \ar[r,"\mu"',Rightarrow] & F_3,
    \end{tikzcd}\]
    where $\chi,\eta$ are in $[\C C,D]$, and $\mu$ is arbitrary. Then for every $c \in \C C$ we get a diagram 
    \[ 
    \begin{tikzcd}
        && F_4(c) \ar[d,"\eta_c"]
        \\
        F_1(c) & F_2(c) \ar[l,"\chi_c"] \ar[r,"\mu_c"'] & F_3(c),
    \end{tikzcd}\]
    where a pullback $\eta_c^*\big(F_2(c)\big)$ exists.
    Thus, from Lemma~\ref{lem::exponent1} a pullback of $\eta$ along $\mu$ exists, and by construction, every arrow lies in $D(F_2(c))$, which allows us to conclude that $\chi_c \circ \mu_c^*\eta_c$ is in $D$, and thus, $\chi \circ \mu^*\eta$ is in $[\C C,D]$.
\end{proof}

\begin{defi}
    Let $[\C C, \C D]_{(C,D)}^{\pull}$ be the subcategory of $[\C C,\C D]_{(C,D)}^{\pull}$ that consists of computability transfers that preserve \emph{all} pullbacks, not only those used in condition (\texttt{Base}\textsubscript{2}).
    Furthermore let $\CatBaseComp^{\pull}$ be the subcategory of $\CatBaseComp$ consisting of all computability models and computability transfers that preserve \emph{all} pullbacks.
\end{defi}

\begin{prop}\label{prp: exp} The computability transfer category \( [\C C, \C D]_{(C,D)}^{\pull} \) with the base \( [\C C, D] \), defined in Proposition~\ref{prp: basefunctorcat}, is the exponential of \( \CB C, \CB D \) in \( \CatBaseComp^{\pull} \).
\end{prop}

\begin{proof}
	Let the following commutative triangle of categories and functors
	\[ \begin{tikzcd}
		{[\C C, \C D] \times \C C} 
		\ar[r,"\eval"] & \C D
		\\
		\C B \times \C C,
		\ar[ur, "F"']
		\ar[u,"\hat F \times 1_\C D"]
	\end{tikzcd} \]
where \(  F \colon \CB B \times \CB C \to \CB D \) is a computability transfer. 
We first show the stronger property that also 
\[ \hat F \colon \C B \to [\C C, \C D]_{(C,D)}^{\pull}.
\] 
To see that $\hat F$ maps $b \in \C B$ to a pullback-preserving computability transfer first observe that if $i \in \C C$, then 
$\big(\hat F(b)\big)(i) = F(1_b,i) \in \C D(F(b,c))$---as $F$ is a computability transfer---and $D(F(b,c)) = D((\hat F(b))(c))$, so 
$\big(\hat F(b)\big)(i) \in D((\hat F(b))(c))$.
To see that $\hat F(b)$ preserves pullbacks assume we are given the following pullback square in $\C C$
\begin{equation}
\begin{tikzcd}
    c_1 \ar[r] \ar[d] \pullback & c_2 \ar[d]
    \\
    c_3 \rar & c_4.
\end{tikzcd}\label{exponen::lem2}
\end{equation}
Then the following is a pullback square in $\C B \times \C C$
\[
\begin{tikzcd}
    (b,c_1) \rar \dar \pullback & (b,c_2) \dar
    \\
    (b,c_3) \rar & (b,c_4).
\end{tikzcd}
\]
Thus it is mapped to a pullback square by $F$. 
But the image of the above pullback square under $F$ is the image of the pullback square~\eqref{exponen::lem2} under $\hat F(b)$, so $\hat F(b)$ is pullback-preserving.
To see that all naturality squares of $\hat F(f)$ are cartesian for every arrow $f\colon b' \to b$ in $\C B$, 
observe that $\hat F(f)_c = F(1_c,f)$ for every $c \in \C C$, thus the naturality squares are of the form 
\begin{equation}
    \begin{tikzcd}
        F(b,c) \ar[r,"{F(g,1_c)}"] \ar[d,"{F(1_b,f)}"'] & F(b',c) \ar[d,"{F(1_{b'},f)}"]
        \\
        F(b,c') \ar[r,"{F(g,1_{c'})}"'] & F(b',c')
    \end{tikzcd}
\end{equation}
and the following are cartesian squares in $\C B \times \C C$
\[
\begin{tikzcd}
    (b,c) \ar[r,"{(g,1_c)}"] \ar[d,"{(1_b,f)}"'] & (b',c) \ar[d,"{(1_{b'},f)}"]
        \\
    (b,c') \ar[r,"{(g,1_{c'})}"'] & (b',c'),
\end{tikzcd}
\]
thus as $F$ is pullback-preserving the claim is shown.

It remains to show that $\hat F$ and $\eval$ are computability transfers.
	To show that \( \hat F \) is a computability transfer, we observe that, if \( g \colon b \to b' \) lies in \( B(c') \), then 
	$	 \hat F(g)_c = F(g,1_c) \in D\big(F(b',c)\big)
	$
    since $F$ is a computability transfer and 
	\[
		\big(	[\C D, \C E],[\C D,E]\big) \times \CB D = \big([\C D, \C E] \times \C D, [\C D,E] \times D\big),	
	\]
    thus $(g,1_c) \in([\C D,E] \times D)(b',c) $.
    That $F$ preserves pullbacks is straightforward to show.
	%as defined earlier, 
	Let \( (\eta,f) \in \big([\C D,E] \times D\big)(G,d) \), i.e., \(
		\eta \colon H \Rightarrow G\) and \(f \colon d' \to d\).
	By definition, \( \eval(\eta,f)  \) is the diagonal in the following commutative diagram
	\[ \begin{tikzcd}
		H(d') \ar[r,"H(f)"] \ar[d,"\eta_{d'}"']
		& H(d) \ar[d,"\eta_d"]
		\\
		G(d') \ar[r,"G(f)"']
		& G(d).
	\end{tikzcd} \]
	By definition of \( [\C D, E]  \) we have that \( \eta_d \in E\big(G(d)\big) \), and since \( H \) is a computability transfer, we get \( \eta_d \circ H(f) \in E\big(G(d)\big) \).
    That $\eta$ preserves pullbacks is straightforward to show.
	Thus, \( \eval  \) is a computability transfer.
\end{proof}

\section{\texorpdfstring{$\CatBaseComp$}{CatBaseComp} is a type-category}\label{sec: typecat}
In this section we present the notion of a \emph{category with a family-arrow structure} and \emph{Sigma-objects}, introduced in~\cite{petrakisCategoriesDependentArrows2023} and directly linked to the notion of \emph{type-category}, introduced by Pitts in~\cite{Pi01}. Sigma-objects generalise the Grothendieck construction in the abstract framework of categories with a family-arrow structure. The main result of this section is that the category $\CatBaseComp$ is a type-category, i.e., a (fam, $\Sigma$)-category with a terminal object (Theorem~\ref{thm: typecat}).
In this way, the category $\CatBaseComp$ can be seen as a model of dependent type theory (see~\cite{hottbook}). 
Moreover in~\cite[Chapter~3]{Ga26} it is shown that (fam, $\Sigma$)-categories correspond to \emph{discrete comprehension categories}.

\begin{defi}\label{def: famCatDef}
    A \emph{fam-category} is a category $\C C$ together with a collection of \emph{family arrows} $\fHom(c)$, for every object $c$ in $\C C$. 
    We denote family arrows by Greek letters $\lambda, \mu,$ etc., and we picture them by arrows starting from $c$
     \[ \begin{tikzcd}
        c \ar[r,"\lambda"] &.
    \end{tikzcd}\]
    For every $c,d \in \C C$ there is a composition operation $\circ \colon \fHom(d) \times \Hom(c,d) \to \fHom(c),$ $ (\lambda,f) \mapsto \lambda \circ f$, such that the following conditions hold:\\[1mm]
  $(F_1)$ For every $c \in \C C$ and $\lambda \in \fHom(c)$ we have $\lambda \circ 1_c = \lambda$
        \[ \begin{tikzcd}
    c \ar[r,"1_c"'] \ar[rr,"\lambda", bend left = 40] & c \ar[r,"\lambda"'] &\none.
\end{tikzcd}\]
$(F_2)$ For every $c,d,e \in \C C$ and $\lambda \in \fHom(e), g \in \Hom(d,e),f \in \Hom(c,d)$ we have that
        $\lambda \circ (g \circ f) = (\lambda \circ g) \circ f$
     %   Diagrammatically this can be visualised through the commutativity of
        \[ \begin{tikzcd}
    c \ar[r,"f"] \ar[rrr,"(\lambda \circ g )\circ f",bend left = 40]
    \ar[rr,"g \circ f", bend right = 40]
    \ar[rrr,"\lambda\circ (g \circ f)"', bend right = 40]
    & d \ar[r,"g"] \ar[rr,"\lambda \circ g"', bend left = 40]
    & e \ar[r,"\lambda"]
    & \none .
\end{tikzcd}\]
\end{defi}

\begin{exas}\label{ex: famarrows}
  %  \begin{enumerate}
  (i) \emph{Constant families}: Every category is turned into a fam-category by setting $\fHom(c) = \C C_0$, the class of objects, for every $c \in \C C$. 
        Composition is defined by $c \circ f = c$ for every $c$ and $f$ in $\C C$.\\
(ii) \emph{Family-arrows in categories}: If $\Cat$ is the category of locally small categories, let $\fHom(\C C) = \Fun(\C C\opp, \Sets)$ with composition the composition of functors.\\
(iii) (Pitts) \emph{Family-arrows in a topos} $\C C$ with subobject classifier $(\top, \Omega)$: If $a \in \C C$, let
%let
        \[ \fHom(a) := \bigcup_{b \in \C C}\Hom(a \times b, \Omega). \]
        The composition is defined by the rule $\big((b,e),g) \mapsto \big(b,e \times (g \times 1_b)\big)$
        %\[ \coprod_{d \in \C C}\Hom(c \times d, \Omega) \times \Hom(b,c) \ni \big((d,g),f) \mapsto \big(d,g \times (f \times \B 1_d)\big).\]
        %Diagramatically this can be visualized as 
        \[\begin{tikzcd}
&c \times b 
	\ar[dr,"g \circ \pr_b"] \ar[dl,"1_b \circ \pr_c"'] 
	\ar[d,"g \times 1_b"{description},dashed] \ar[dd,"{(b,e) \circ g}"{near start}, to path = {
	(\tikztostart.north) |-
	([xshift = 6em, yshift = .5em]\tikztostart.north) 
	|-  (\tikztotarget.east) \tikztonodes }, rounded corners]
\\
a 
 & a \times b  
	\ar[d,"e"'] \ar[l,"\pr_a"] \ar[r,"\pr_b"'] 
& b 
\\
& \Omega.
\end{tikzcd} \]
(iv) If $A$ is a type in a universe $U$ (see~\cite{hottbook}), a family arrow on $A$ is a type-family $P \colon A \to U$.
%over $A$.
\end{exas}

\begin{defi}
 \( \CatBaseComp\) is a fam-category where for each \( \CB C \) we define
 \[
 	\fHom(\CB C) := \{ F \colon \C C\opp \to \Sets \}.
 \]
The composition is simply the composition of the underlying functors.
\end{defi}

\begin{defi}\label{def: famSigma}
A fam-category $\C C$ has \emph{Sigma-objects}, or is a (fam, $\Sigma$)-\emph{category}, if for every $a, b \in \C C$, for every $\lambda \in \fHom(a)$, and for every $f \in \Hom(b,a)$ there is a \emph{Sigma-object} $\sum_a \lambda \in \C C$ and
%there are 
arrows $\pr_1^{a, \lambda} \in \Hom\big(\sum_a \lambda, a\big)$ and $\Sigma_{\lambda}f \in  \Hom\big(\sum_b (\lambda \circ f), \sum_a \lambda\big)$, such that the following rectangle is a pullback, 
	\[
    \begin{tikzcd}
        \Grothendieck{b}{(\lambda \circ f)} \pullback\ar[r,"\Grothendieck{\lambda}{f}"]
        \ar[d,"\pr_1^{b,\lambda \circ f}"']
        &  \Grothendieck{a}{\lambda} \ar[d,"\pr_1^{a,\lambda}"]
        \\
        b \ar[r,"f"'] & a
    \end{tikzcd}
    \]
and the following conditions hold:\\[1mm]
$(\Sigma_1)$ $\Sigma_{\lambda}1_a = 1_{ \sum_a \lambda}$,\\
$(\Sigma_2)$ $\ \Sigma_{\lambda}(f \circ g) = \big(\Sigma_{\lambda}f\big) \circ \Sigma_{(\lambda \circ f)}g$, for every $f \in \Hom(b, a)$ and $g \in \Hom(c, b)$
	\[
    \begin{tikzcd}
        \Grothendieck{a}{(\lambda \circ 1_a} \pullback\ar[r,"\Grothendieck{\lambda}{1_a}"]
        \ar[d,"\pr_1^{a,\lambda \circ 1_a}"'] 
        & \Grothendieck{a}{\lambda} \ar[d,"\pr_1^{a,\lambda}"]
        \\
        a \ar[r,"1_a"'] & a
    \end{tikzcd}
    \]
	\[
    \begin{tikzcd}
        \Grothendieck{c}{(\lambda \circ f) \circ g} \pullback\ar[r,"\Grothendieck{(\lambda \circ f)}{g}"]
        \ar[d,"\pr_1^{c,(\lambda \circ f) \circ g}"']
        \ar[rr,"\Grothendieck{\lambda}{(f \circ g)}", bend left = 30]
        &[2em] \Grothendieck{b}{(\lambda \circ f)} \pullback
        \ar[r,"\Grothendieck{\lambda}{f}"]
        \ar[d,"\pr_1^{b, \lambda \circ f}"']
        &[2em] \Grothendieck{a}{\lambda}
        \ar[d,"\pr_1^{a,\lambda}"]
        \\
        c \ar[r,"g"'] 
        & b \ar[r,"f"']
        & a.
    \end{tikzcd}
    \]
A \emph{type-category} (Pitts~\cite{Pi01}) is a (fam, $\Sigma)$-category with a terminal object\footnote{In~\cite{Pi01}, pp.~110-111, Pitts does not study the family-arrow structure of a type-category separately from its $\Sigma$-structure. In~\cite{petrakisCategoriesDependentArrows2023} examples of (fam, $\Sigma)$-categories without a terminal object are given.}.
\end{defi}

\begin{exas}\label{ex: famSigma}
(i) If $\C C$ is a category with constant families $\fHom(c) = \C C_0$ for every $c \in \C C$, then let $\Grothendieck{c}{d} := c$ and $\pr_1^{c,d} = 1_c$. The square
        \[ \begin{tikzcd}
            c \ar[r,"f"] \ar[d,"1_c"'] \pullback 
            & d \ar[d,"1_d"] 
            \\
            c \ar[r,"f"'] & d
        \end{tikzcd}\]
        is a pullback and conditions $(\Sigma_1, \Sigma_2)$ hold trivially.\\
(ii) If $P \in \fHom(\C C) = \Fun(\C C\opp,\Sets)$, the Sigma-object over $P$ is the category of elements $\Grothendieck{\C C}{P}$.\\
(iii) For the definition of the canonical Sigma-objects in a topos, see~\cite{Pi01}, p.~113.\\
(iv) The Sigma-object in the category of small types $\C U$ over a type $A$ in $\C U$ and a type family $P \colon A \to \C U$ is the dependent-pair type $\sum_{x \colon A}P(x)$ (see~\cite{hottbook}).
\end{exas}

%\begin{cor} 
%	If we are given an opfibration \( F \colon \C E \to \C B \) and we have a base \( B  \) on \(  \C B \), then we obtain a base \( B_F \) on \( \C E \) in the following manner:
%	\[
%		B_F(e) = \big\{ i \in \Mon(-,e) ~\vert~ i \text{ cartesian lift of } j \in B\big(F(e)\big)\big\}.
%	\]
%\end{cor}
%
%\begin{proof} 
%	Immediate from the preceding proposition as an opfibration \( F \colon \C E \to \C B \) is a fibration \( F \colon \C E\opp \to \C B\opp \).
%\end{proof}

\begin{lem}\label{lem: GrothendieckBase}
If \( P \in \fHom(\CB C) \), then the projection \( \pr_1^P \colon \big( \Grothendieck{\C C}{P}, (\pr_1^P)^{-1}(C)\big) \to \CB C \) is a computability transfer.
%\[
%	\pr_1^P \colon \Big( \Grothendieck{\C C}{P}, C_{\pr_1^P}\Big) \to \CB C.
%\]
\end{lem}

\begin{proof}
	We apply Corollary~\ref{cor: FibLiftCompTrans}, which is possible, since $\pr_1^{P}$ is a Grothendieck fibration.
\end{proof}

\begin{lem}\label{lem: terminal}
    \(\CatBaseComp\) has a terminal object.
\end{lem}

\begin{proof}
    A terminal object in $\CatBaseComp$ is given by the terminal category $\bbone$, and the base of computability on it contains only this identity arrow.
    It is immediate to show that this is a terminal object in $\CatBaseComp$, as for every other category with a base of computability there exists only one functor $!_{\C C} \colon \C C \to \bbone$, since $\bbone$ is the terminal object of $\Cat$, and this functor is, trivially, a computability transfer.
\end{proof}

\begin{thm}\label{thm: typecat}
 \( \CatBaseComp \) is a type-category.
 %(fam,$\Sigma$)-category.
\end{thm}

\begin{proof}
By Lemma~\ref{lem: terminal}, it suffices to show that \( \CatBaseComp \) is a (fam, $\Sigma$)-category. 
To show that the following rectangle is a pullback   
\[ \begin{tikzcd}
	\Big( \Grothendieck{\C D}{P \circ F}, (\pr_1^{P \circ F})^{-1}(D)\Big) \ar[r,"\Grothendieck{P}{F}"]  \ar[d,"\pr_1^{P \circ F}"']
	& \Big( \Grothendieck{\C C}{P}, (\pr_1^P)^{-1}(C)\Big) \ar[d,"\pr_1^P"] 
	\\
	\CB D \ar[r,"F"'] & \CB C,
\end{tikzcd} \]
we simply remark that the following rectangle is a pullback 
 \[ \begin{tikzcd}
	 \Grothendieck{\C D}{P \circ F} \ar[d,"\pr_1^{P \circ F}"'] \ar[r,"\Grothendieck{P}{F}"] 
	 & \Grothendieck{\C C}{P} \ar[d,"\pr_1^P"]
	 \\
	 \C D \ar[r,"F"'] & \C C,
 \end{tikzcd} \]
% is a pullback square in \( \Cat \), 
%and we use Theorem~\ref{thm: PullbackFibLift}. 
and we use Theorem~\ref{thm: PullbackFibLift}.
The strictness conditions are inherited from the strictness conditions in \( \Cat \).
\end{proof}

\section{\texorpdfstring{$\CatBaseComp$}{CatBaseComp} is a (2-fam, \texorpdfstring{\( \Sigma \)}{Sigma})-category}
\label{sec: 2famS}
%\section{$\CatBaseComp$ is a (2-fam, \( \Sigma \))-category}

In this section we show that \( \CatBaseComp \) is a (2-fam, \( \Sigma \))-category (Proposition~\ref{prp: 2famSigma}), a 2-categorical generalisation of a (fam, $\Sigma$)-category, studied by Ehrhardt in~\cite{ehrhardt2depCategories2024}.
In~\cite[Chapter~3]{Ga26} it is shown that (2-fam, $\Sigma$)-categories correspond to \emph{split comprehension categories}.

\begin{defi}\label{def: 2fam}
A fam-category $\C C $ is a 2-fam-\textit{category}, if for every $c \in \C C$ the collection $\fHom(c)$ is a category 
%whose objects are the family arrows and 
whose morphisms are called \emph{2-family arrows}.        
A 2-family arrow $\eta \in \Hom(\lambda, \mu)$ is pictured as follows: 
\[ \begin{tikzcd}
c \ar[r,"\lambda"{name = U}, bend left = 40] \ar[r,"\mu"'{name = V}, bend right = 40]
\ar[from = U, to = V, Rightarrow,"\eta"] 
&\phantom{\cdot}.
\end{tikzcd}\]
Moreover, for every $c,c' \in \C C,\lambda,\mu \in \fHom(c')$ there is an operation $\bullet_{c,c'} $ assigning to $\eta \in \Hom(\lambda,\mu)$ and $f \colon c \to c'$ the 2-family-arrow $\eta \bullet_{c,c'} f \in \Hom(\lambda \circ f, \mu \circ f)$, such that the following conditions hold:\\[1mm] 
\emph{Compatibility:} For every $c,c',c'' \in \C C$, $\lambda,\mu \in \fHom(c'')$, $\eta \in \Hom(\lambda,\mu)$ and every $f \in \Hom(c,c'), g \in \Hom(c',c'')$ we have that 
\renewcommand{\arraystretch}{6}
        \\
        \begin{tabular}{lp{0.8\textwidth}}
       \parbox{0.3\textwidth}{ $\eta \bullet_{c,c} 1_c = \eta $ }
        & 
        $\begin{tikzcd}
    c \ar[r,"1_c"] 
    \ar[rr,"\lambda"{name = U, near end,swap},  to path = 
    {(\tikztostart.north) |- ([yshift = 2em]\tikztotarget.north) \tikztonodes -- (\tikztotarget.north) }, rounded corners]
    \ar[rr,"\mu"'{name = V,near end,swap},to path = 
    {(\tikztostart.south) |- ([yshift = -2em]\tikztotarget.south) \tikztonodes -- (\tikztotarget.south) }, rounded corners]
    & c \ar[r,"\lambda"{name = U2}, bend left = 40] \ar[r,"\mu"'{name = V2},bend right = 40] \ar[from = U2, to = V2, Rightarrow,"\eta"]
    & \phantom{\cdot} 
    \ar[from = U, to = V,"\eta \bullet_{c,c} 1_c"'{near end}, crossing over,Rightarrow,to path =
    { (\tikztostart.north) -- ([yshift = 1em]\tikztostart.north) -| ([xshift = -5em,yshift = -1em]\tikztotarget.south) \tikztonodes -|(\tikztotarget.south)}, rounded corners] 
    \end{tikzcd}$
\end{tabular}\\
\begin{tabular}{lp{0.8\textwidth}}
\parbox{0.3\textwidth}{$\eta \bullet_{c,c''} (g \circ f) $\\
$= (\eta \bullet_{c',c''} g) \bullet_{c,c'} f $ } & $
\begin{tikzcd}
    c \ar[r,"f"] \ar[rrr,"\lambda \circ g \circ f"'{name = U2, near end}, 
    to path = {(\tikztostart.north) |- ([yshift = 2em]\tikztotarget.north) \tikztonodes -- (\tikztotarget.north) }, rounded corners]
    \ar[rrr,"\mu \circ g \circ f"{name = V2,near end}, to path = {(\tikztostart.south) |- ([yshift = -2em]\tikztotarget.south) \tikztonodes -- (\tikztotarget.south) }, rounded corners]
    & c' \ar[r,"g"]
    & |[text height = 4pt]| c'' 
    \ar[r,"\lambda"{name = U}, bend left = 40, start anchor = {north east}] 
    \ar[r,"\mu"'{name = V}, bend right = 40, start anchor = {south east}]
    &  \phantom{\cdot}.
    \ar[from = U, to = V, Rightarrow, "\eta"]
    \ar[from = U2, to = V2, Rightarrow,"\eta \bullet_{c'',c} (g \circ f)"{near end},
    to path = {(\tikztostart.north) -- ([yshift = 1em]\tikztostart.north)-| ([xshift = 8em,yshift = -1em]\tikztotarget.south) \tikztonodes -| (\tikztotarget.south) }, rounded corners] 
\end{tikzcd} $
\end{tabular}\vspace{1em}
        \item\emph{Distributivity:}
        For every $c,c' \in \C C$, $\lambda,\delta,\gamma \in \fHom(c')$, $\eta\in \Hom(\delta,\gamma), \eta' \in \Hom(\gamma,\lambda)$, and every $f \in \Hom(c,c')$ we have that 
        \[ (\eta' \circ \eta) \bullet_{c,c'} f = (\eta' \bullet_{c,c'} f) \circ (\eta \bullet_{c,c'} f)\qquad\begin{tikzcd}
    c \ar[r,"f"] 
    & |[text height = 5pt]| c' \ar[r,"\lambda"{name = U, near end},  to path = {(\tikztostart.north) |- ([yshift = 2em]\tikztotarget.north) \tikztonodes -- (\tikztotarget.north) }, rounded corners]
    \ar[r,"\delta"{name = V,description}]
    \ar[r,"\gamma"'{name = W, near end}, to path = {(\tikztostart.south) |- ([yshift = -2em]\tikztotarget.south) \tikztonodes -- (\tikztotarget.south) }, rounded corners]
    & \none.
    \ar[from = U, to = W, "\eta"{near start}, Rightarrow, shorten >= 35pt]
    \ar[from = U, to = W, "\eta'"{near end}, Rightarrow, shorten <= 35pt]
    %\ar[from = V, to = W, "\eta'", Rightarrow]
\end{tikzcd}\]
\end{defi}

\begin{rem}\label{rem: omit}
     Usually, we omit the indices in $\bullet_{c,c'}$ and we only write $\bullet$. 
    The operation $\bullet$ is called ``horizontal composition''.
\end{rem}

\begin{exas}\label{ex: 2famarrows}
(i) If $M$ is a monoid, and $\C C$ is a fam-category, we define a 2-fam-structure on $\C C$ by letting $\Hom(\lambda, \mu) := M$ for $\lambda,\mu \in \fHom(c)$ and $c \in \C C$. 
Compositions $\bullet$ are defined by the rule $m \bullet f := m$. The compatibility and distributivity properties are immediate to show. \\
(ii) All fam-categories in Examples~\ref{ex: famarrows} have their 2-analogue (see~\cite[$\S$3.1]{ehrhardt2depCategories2024}). 
\end{exas}

Next we define the Sigma-structure that corresponds to a 2-fam-category.

\begin{defi}\label{def: 2famSigma}
    A 2-fam-category $\C C$ is a (2-fam, $\Sigma$)-\emph{category} if the underlying fam-category has a (fam,$\Sigma$)-structure and for every $c \in \C C$, $\lambda,\mu \in \fHom(c)$ and every $\eta \in \Hom(\lambda,\mu)$ there is an arrow 
    \[ \sum_{\lambda,\mu} \eta \colon \sum_{c} \lambda \to \sum_{c} \mu, \]
    such that the following diagrams commute 
    \[ \begin{tikzcd}
        \sum_{c} \lambda \circ f 
        \ar[r,"\sum_{\lambda} f"]
        \ar[d,"\sum_{\lambda \circ f,\mu \circ f} \eta\bullet f"']
        \ar[dd, to path = {
        (\tikztostart.west) -- ([xshift = -5.5em]\tikztostart.west) |- (\tikztotarget.west) \tikztonodes 
        }, "\pr_1^{c,\lambda \circ f}"'{near start},rounded corners]
        &[2em] \sum_{c'} \lambda 
        \ar[d,"\sum_{\lambda,\mu} f"]
        \ar[dd, to path = {
        (\tikztostart.east) -- ([xshift = 5.5em]\tikztostart.east) |- (\tikztotarget.east) \tikztonodes }, "\pr_1^{c',\lambda}"{near start},rounded corners] 
        \\
        \sum_{c} \mu \circ f 
        \ar[r,"\sum_\mu f"'] 
        \ar[d,"\pr_1^{c,\mu \circ f}"']
        & \sum_{c'} \mu \ar[d,"\pr_1^{c',\mu}"]
        \\
         c \ar[,r,"f"'] & c'
    \end{tikzcd}\]
Moreover, the following strictness conditions hold:
    \[ \sum_{\lambda,\lambda} 1_\lambda = 1_{\sum_{c} \lambda} \qquad\hbox{and}\qquad 
    \sum_{\lambda,\mu} \eta \circ \sum_{\nu,\lambda} \eta' = \sum_{\nu,\mu} (\eta \circ \eta'). 
    \]
\end{defi}
\begin{prop}\label{prp: 2famSigma}
	The category \( \CatBaseComp \) is a (2-fam, \( \Sigma \))-category in the following way:
	the 2-family-arrows are natural transformations and the horizontal composition is the horizontal composition of a functor and a natural transformation.
	The arrows 
	\[
		\Grothendieck{P,S}{\eta} \colon \Grothendieck{\CB C}{P} \to \Grothendieck{\CB C}{S}
	\]
	send \( (c,y) \) to \( \big(c,\eta_c(y)\big) \).
\end{prop}

\begin{proof} 
It is immediate to show that with this definition of 2-fam-arrows compatibility and distributivity are satisfied, hence, $\CatBaseComp$ is a 2-fam-category. Next, we show that \( \Grothendieck{P,S}{\eta} \) is a computability transfer.
	If \( i \colon (c,y) \to (c',y') \in \Grothendieck{\C C}{C}(c',y')\), then by definition $i \colon c \to c' \in C(c')$, so $\Grothendieck{P,S}{\eta}(i) = i \in \Grothendieck{\C C}{C(c',\eta_{c'}(y)}$.
Since \(\pr_1^P(c,y) = c = \pr_1^S(c,\eta_c(y)) = \pr_1^S\Big(\Grothendieck{P,S}{\eta}(c,y)\Big)\), the following diagram commutes
%	The commutativity of the following triangle
	\[ \begin{tikzcd}
		\Grothendieck{\CB C}{P} \ar[rr,"\Grothendieck{P,S}{\eta}"] 
		\ar[dr,"\pr_1^P"'] &&
		\Grothendieck{\CB C}{S} 
		\ar[dl,"\pr_1^S"]
		\\
				  & \CB C,
	\end{tikzcd}  \]
hence, $\Grothendieck{P,S}{\eta}$ is well-defined, if $\eta \colon P \To S$.
    The strictness conditions are shown as follows:
    \begin{align*}
        &\Grothendieck{P,P}{1_P}(c,x) = (c,(1_P)_c(x)) = (c,x) = 1_{\Grothendieck{\C C}{P}}(c,x), &&\Grothendieck{P,P}{1_P}(i) = i = 1_{\Grothendieck{\C C}{P}}(i),
        \\
        &\Grothendieck{P,T}{\mu \circ \eta}(c,x) = \big(c,\mu_c(\eta_c(x))\big) = \Grothendieck{S,T}{\mu}\Big(\Grothendieck{P,S}{\eta}(c,x)\Big),
        &&\Grothendieck{P,T}{\mu \circ \eta}(i) = i = \Grothendieck{S,T}{\mu}\Big(\Grothendieck{P,S}{\eta}(i)\Big). \ \ \ \ \qedhere
    \end{align*}
%	can be computed via 
%		\pr_1^P(c,y) = c = \pr_1^S(c,\eta_c(y)) = \pr_1^S\Big(\Grothendieck{P,S}{\eta}(c,y)\Big).
%	\end{align*}
\end{proof}

\section{\texorpdfstring{\( \CatBaseComp \) is a (2-dep, $\Sigma$)-category}{CatBaseComp is a (2-dep,Sigma)-category}}
\label{sec: 2depS}

A dependent arrow is an abstract categorical formulation of the type-theoretic notion of a dependent function. The axioms of a category endowed with a dependent-arrow structure capture the properties of composition of a dependent function with a function, exactly as the axioms of a category capture the properties of compositions of functions.
In \cite{petrakisCategoriesDependentArrows2023} it is shown that a (fam, $\Sigma$)-category can be equipped with a canonical dependent arrow-structure, turning it into a (dep, $\Sigma$)-category.
In section~\ref{sec: 2depS} we present the notion of a (2-dep, $\Sigma$)-category, and describe the dependent arrows obtained by the canonical construction, applied to the (2-fam, $\Sigma$)-structure on \( \CatBaseComp \), presented in section~\ref{sec: 2famS}.
 In~\cite[Chapter~7]{Ga26} it is shown that (2-dep, $\Sigma$)-categories correspond to the introduced \emph{higher comprehension categories}.

\begin{defi}\label{def: depSigma}
	Let \( \C C \) be a (fam, $\Sigma$)-category. 
A \emph{dep-structure} on \( \C C \) consists of a collection \( \dHom(c,\lambda) \) for every object \( c \) of \( \C C \) and every \( \lambda \in \fHom(c)\), together with a composition operation \( \circ \) that maps \( \Phi \in \dHom(c,\lambda) \) and \( f \in \Hom(c',c) \) to \( \Phi \circ f \in \dHom(c',\lambda \circ f) \), and a \emph{second-projection-arrow} \( \pr_2 \in \fHom(\lambda \circ \pr_1^{\lambda}) \), for every family arrow \( \lambda \), such that the following conditions hold:\\[1mm]
$(D_1)$ \( \Phi \circ (g \circ f) = (\Phi \circ g) \circ f \) and \( \Phi \circ 1_c = \Phi \), for every \( \Phi,g,f \).
\[ 	 \begin{tikzcd}
		c'' \ar[r,"f"] & c' \ar[r,"g"] & c \ar[r,"\lambda"{near end, name = L}]
			       & \none
		       \ar[from = 1-2,to = L, "{\Phi\circ g}", bend left = 30,Rightarrow]
		       \ar[from = 1-1, to = L, "{\Phi\circ (g\circ f)}"',end anchor = {south}, bend right = 30, shorten >= 2pt, Rightarrow]
	\end{tikzcd}\quad\hbox{and}\quad\begin{tikzcd}
		c \ar[r,"1_c"] & c \ar[r,"\lambda"{near end, name = L}]
			       & \none
			       \ar[from = 1-1, to = L,"\Phi", to path = 
			       {(\tikztostart.north)|- ([yshift = 1em]\tikztotarget.north) \tikztonodes -- (\tikztotarget.north)}, rounded corners,Rightarrow]
			       \ar[from = 1-2,to = L, "\Phi", bend left = 30,Rightarrow]
	\end{tikzcd} 
 \]
	\\[1mm]
$(D_2)$ \( \pr_2^{\lambda \circ f} = \pr_2^{\lambda} \circ \Grothendieck{\lambda}{f} \), for every \( \lambda \in \fHom(c)	\) and \( f \colon c' \to c \).
\[ \begin{tikzcd}
			\Grothendieck{a}{\lambda \circ f} \ar[rr,"\Grothendieck{\lambda}{f}"] 
			\ar[dr,"\lambda \circ f \circ \pr_1^{\lambda \circ f}"'{name = U}] 
			&& \Grothendieck{b}{\lambda} \ar[dl,"\lambda \circ \pr_1^\lambda"{name = V}]
			\\
			& {\cdot}
			\ar[from = 1-3, to = V,"\pr_2^{\lambda}",start anchor = {[xshift = -7pt]south east}, bend left = 30,Rightarrow]
			\ar[from = 1-1, to = U,"\pr_2^{\lambda\circ f}"', start anchor = {[xshift = 7pt]south west},bend right = 30, Rightarrow]
		\end{tikzcd} \]
        \\[1mm]
We call a (fam, $\Sigma$)-category together with a dependent arrow-structure a (dep, $\Sigma$)\emph{-category}.
\end{defi}

As shown in~\cite{petrakisCategoriesDependentArrows2023}, every (fam, $\Sigma$)-category is turned into a (dep,$\Sigma$)-category by defining the dependent arrows as the sections of the first-projection-arrow
\[ \begin{tikzcd}
	c \ar[r,"\Phi"] \ar[dr,"1_c"']  & \Grothendieck{c}{\lambda} \ar[d,"\pr_1^\lambda"]
	\\
					& c.
\end{tikzcd} \]
The second-projection-arrow and the composition are defined via pullbacks as follows:
\[ \begin{tikzcd}
	\Grothendieck{c}{\lambda} \ar[drr,"\Phi \circ \pr_1^\lambda"{near start},to path = { (\tikztostart.east) -| (\tikztotarget.north)\tikztonodes}, rounded corners]
	\ar[ddr,"1_{\Grothendieck{c}{\lambda}}"{swap, near start},to path = { (\tikztostart.south) |- (\tikztotarget.west) \tikztonodes}, rounded corners]
	\ar[dr,"\pr_2^{\lambda}"{description},dotted]
	\\
	& \Grothendieck{\Grothendieck{c}{\lambda}}{\lambda \circ \pr_1^\lambda}\ar[d,"\pr_1^{\lambda \circ \pr_1^\lambda}"']
	\ar[r,"\pr_1^{\lambda \circ \pr_1^{\lambda}}"] & \Grothendieck{c}{\lambda} \ar[d,"\pr_1^{\lambda}"]
	\\
						       &\Grothendieck{c}{\lambda}\ar[r,"\pr_1^\lambda"']  & c 
\end{tikzcd}\hbox{ and } \begin{tikzcd} 
%	c'' \ar[r] \ar[ddr,to path = {(\tikztostart.south) |- (\tikztotarget.west) \tikztonodes}, rounded corners]
%	\ar[dr,"\Phi \circ (f \circ g)"{description}]
	 c' \ar[drr,"\Phi \circ f"{near start}, to path = { (\tikztostart.east) -| (\tikztotarget.north) \tikztonodes}, rounded corners]
	 \ar[ddr,"1_{c'}"{swap, near start},to path = { (\tikztostart.south) |- (\tikztotarget.west) \tikztonodes}, rounded corners]
	\ar[dr,"\Phi \circ f"{description},dotted]
	\\
%	&\Grothendieck{c''}{\lambda \circ (f \circ g)} \ar[d,"\pr_1^{\lambda \circ (f \circ g)}"'] \ar[r,"\Grothendieck{\lambda \circ f}{g}"] 
	&\Grothendieck{c'}{\lambda \circ f} \ar[d,"\pr_1^{\lambda \circ f}"'] \ar[r,"\Grothendieck{\lambda}{f}"]
	& \Grothendieck{c}{\lambda} \ar[d,"\pr_1^{\lambda}"]
	\\
%	&c'' \ar[r,"g"]
	& c' \ar[r,"f"'] & c.
\end{tikzcd}
\]
Thus, we can endow \( \CatBaseComp \) with dependent arrows that are computability transfers \( \Phi \colon \CB C \to \big(\Grothendieck{\C C}{F}, \Grothendieck{\C C}{C}\big) \).
In \cite{ehrhardt2depCategories2024} it is shown that the above construction can be extended to (2-fam, $\Sigma$)-categories, to yield (2-dep, $\Sigma$)-categories. 

\begin{defi}\label{def: 2depSigma}
	If \( \C C \) is a (dep, $\Sigma$)-category with a 2-family-arrow-structure, then \( \C C  \) is a \emph{(2-dep, $\Sigma$)-category} if we can form the post-composition \( \eta \circ \Phi \in \dHom(c,\nu) \), for every \( \Phi \in \dHom(c,\lambda)\) and \(\eta \colon \lambda \To \nu \), such that the following conditions hold:\\[1mm]
    $(2D_1)$ \( 1_\lambda \circ \Phi = \Phi \) and \( (\theta \circ \eta) \circ \Phi = \theta \circ (\eta \circ \Phi) \), for \( \Phi \in \dHom(c,\lambda), \eta \colon \lambda \To \mu, \theta \colon \mu \To \nu \).\\[1mm]
    $(2D_2)$ \((\eta \circ \pr_1^{\lambda}) \circ  \pr_2^{\lambda} =\pr_2^\mu \circ \Grothendieck{\lambda,\mu}{\eta}\), for \( \eta \colon \lambda \To \mu \).\\[1mm]
    $(2D_3)$ \( (\eta \circ \Phi) \circ f = (\eta \bullet f) \circ (\Phi \circ f)\) for all $\eta \colon \lambda \To \mu, \Phi \in \dHom(c,\lambda)$ and $f \colon c' \to c$. 
\end{defi}

The post-composition for the canonical dep-structure is defined by the rule
\[
	\eta \circ \Phi := \Grothendieck{\lambda,\mu}{\eta \circ \Phi}.
\]
Note that the compositions on either side of the above equality are different! The canonical dependent arrows in $\CatBaseComp$ 
%the computability transfers obtained in this way 
share the following interesting property.

\begin{lem}\label{lem: interesting}
    If $\Phi,\Psi$ are canonical dependent arrows in $\CatBaseComp$, such that $\Phi(c) = \Psi(c)$, for every object $c$ of $\C C$,then 
    \[ \dom(\Phi(f)) = \dom(\Psi(f)),\]
    for every arrow $f$ in the base of computability $C$ in $\C C$.
\end{lem}

\begin{proof}
    Let $c$ in $\C C$ and $e \in P(c)$, such that $\Phi(c) = (c,e) = \Psi(c)$.
    Since both $\Phi(f)$ and $\Psi(f)$ must be cartesian for $f$, where $f \in C(c)$, we get 
    \[\Phi(f) = \Psi(f) = (f \colon \big((\dom(f),P(f)(e)) \to (c,e)\big).\qedhere\]
\end{proof}

\begin{rem}
    Note that the above result does not rely on the choice of our base of computability. Hence, the canonical arrows in $\Cat$ satisfy this property.
\end{rem}

%\section{Transporting bases of computability}\label{sec: transport}
%\input{Transporting_bases_of_computability}

%\section{Canonical computability models}\label{sec: canonical}
%\input{Canonical_computability_models}

%\section{The category of assemblies}

%\input{Assemblies}

%\section{Computability models and the Grothendieck construction}

%\input{Grothendieck}

\section{Concluding comments}\label{sec: concl}

In this paper we defined and studied the category $\CatBaseComp$ of categories with a base of computability and computability transfers.
Categories with a base of computability were introduced in~\cite{Pe22}, in order to generate computability models in a canonical way.
Having defined here the appropriate notion of arrow between these categories, one can induce from a computability transfer between two appropriate categories with a base of computability a simulation between the associated computability models. Specifically, if 
$(\C{C},C)$ and $(\C{D},D)$ are categories with a base of computability, and if $S \colon \mathscr{C} \to \Sets$,
$R \colon \mathscr{D} \to \Sets$ are pullback-preserving presheaves, let  $F \colon (\C{C},C) \to (\C{D},D)$ be a computability transfer, such that the following triangle commutes
\[ \begin{tikzcd}
\mathscr{C} \ar[rr,"F"] \ar[dr,"S"'] && \mathscr{D} \ar[dl,"R"] 
\\
& \Sets.
\end{tikzcd} \]
Following the definition of a partial computability model in Example~\ref{ex: base}, there is a simulation $\pmb{\gamma} \colon \mathbf{CM}^C(\mathscr{C};S) \rightarrowtriangle \mathbf{CM}^D(\mathscr{D};R)$ 
defined as follows: the class function $\gamma \colon \mathscr{C}_0 \to \mathscr{D}_0$ is defined by $\gamma := F_0$, and for every $c \in \mathscr{C}_0$ we define $\Vdash_c^\gamma \ \subseteq R(F(c)) \times S(c)$ by the rule
	\[ y \Vdash_c^\gamma x :\Leftrightarrow y = x  \]
\[ \begin{tikzcd}
&\mathscr{C} \ar[dd,"F"{name = U}] \ar[dl,"S"'] &\mathbf{CM}^C(\C{C};S) \ar[dd,"\pmb{\gamma}"' 
{name = V},-simto,"="{sloped}]
\\
\Sets 
\\
&\mathscr{D} \ar[ul,"R"] &\mathbf{CM}^D(\C{D};R).
\arrow[from = U, to = V, mapsto]
\arrow[from = U, to = \tikzcdmatrixname-2-1, "\circlearrowright" description, phantom]
\end{tikzcd} \]
Using the definitions in~\cite{petrakisStrictComputabilityModels2022},  $\pmb{\gamma}$ is an equality simulation, although, not always a full one. If we fix a category with base of computability $(\C{C}, C)$ and
$S,S' \colon \C{C} \to \Sets$ are pullback-preserving presheaves on $\C{C}$, then a natural transformation 
$\mu \colon S \Rightarrow S'$ induces a simulation $\pmb{\gamma}^\eta = (\id_{\C{C}_0}, 
(\Vdash_c^{\gamma^\mu})_{c \in \C{C}_0}) \colon \mathbf{CM}^B(\C{C};S) \rightarrowtriangle 
\mathbf{CM}^B(\C{C};S')$,
\[ \begin{tikzcd}
\mathscr{C} \ar[dd,"S"'{name = U},bend right = 40] \ar[dd,"S'"{name = V}, bend left = 40] 
&\mathbf{CM}^B(\C{C};S) \ar[dd,"\pmb{\gamma}^\mu"'{name = W},-simto,"\ast"{sloped}]
\\
\phantom{0}
\\
 \Sets &\mathbf{CM}^B(\C{C};S').
\arrow[from = U, to = V, Rightarrow,shorten <= 2pt, shorten >= 2pt, "\mu"{name = X}]
\arrow[from = X, to = W,mapsto, bend left = 25,crossing over]
\end{tikzcd} \]
where $\Vdash_c^{\gamma^\mu} \subseteq S'(c) \times S(c)$
 is defined through the equivalence
\[ y \Vdash_c^{\gamma^\mu} x :\Leftrightarrow y = \mu_{S(c)}(x), \]
and ``$\ast$'' above denotes that $\pmb{\gamma}^\mu$ is a
natural simulation, a notion introduced in~\cite{petrakisStrictComputabilityModels2022}. Notice that 
not every natural simulation between canonical computability models 
$\CM^B(\C{C};S)$ and $\CM^{B}(\C{C};S')$ is induced by some natural transformation (see~\cite{Ga26}). 

The connection established in sections~\ref{sec: typecat},~\ref{sec: 2famS}, and~\ref{sec: 2depS} between the category $\CatBaseComp$ and Dependent Type Theory has the following natural extension. All (dependent) categorical notions introduced in these sections have their immediate dual (codependent) counterpart (see~\cite{petrakisCategoriesDependentArrows2023}). The \emph{cofamily arrows} on an object $c \in \C{C}$ are (codependent) objects that are composed with arrows $f \colon c \to d$ in $\C{C}$, satisfying the dual conditions of $(F_1, F_2)$ in Definition~\ref{def: famCatDef}. The dual of a 
$Sigma$-object in a fam-category is that of a \emph{co}$Sigma$-object, or a \emph{quotient-oblject} in a cofam-category\footnote{The duality between (fam, $\Sigma)$-categories and (cofam, co$\Sigma)$-categories corresponds to the duality between comprehension and quotient structures
in~\cite{MR20}.}. Projection-one-arrows are dualised by \emph{injection-one-arrows}, and the fundamental pullback in Definition~\ref{def: famSigma} is dualised by the corresponding pushout. \emph{Codependent arrows} are the dual of the dependent arrows, and the \emph{second-injection-arrow} is the dual of the second-projection-arrow in Definition~\ref{def: depSigma}. All examples and results about 
fam- or dep-categories are dualised immediately for cofam- or codep-categories (the same holds for $2$-fam- and $2$-dep-categories). The formulation of the interplay between the dependent-arrow-structure and the codependent-arrow-structure on a category is a major future task.
In subsequent work we plan to investigate possible canonical cofam- and codep-structures on $\CatBaseComp$, and their relation to the fam- and dep-structures on $\CatBaseComp$ that were introduced here.

\printbibliography

\end{document}